\documentclass[10pt]{amsart}

\usepackage{multirow}
\usepackage{amsmath,fixltx2e}
\usepackage{amsthm}
\usepackage{amssymb}
\usepackage{amsfonts}
\usepackage{mathtools}
\usepackage{MnSymbol}

\newcommand{\overbar}[1]{\mkern 1.5mu\overline{\mkern-1.5mu#1\mkern-1.5mu}\mkern 1.5mu}

\newtheorem{theorem}{Theorem}[section]
\newtheorem{corollary}[theorem]{Corollary}
\newtheorem{proposition}[theorem]{Proposition}
\newtheorem{lemma}[theorem]{Lemma}
\newtheorem{conjecture}[theorem]{Conjecture}
\newtheorem{definition}[theorem]{Definition}
\newtheorem{remark}[theorem]{Remark}
\newtheorem{assumption}[theorem]{Assumption}
\usepackage[utf8]{inputenc}
\usepackage[legalpaper, portrait, margin=1.5in]{geometry}
\usepackage{graphicx}
\usepackage{hyperref}
\usepackage{tikz-cd} 
\tikzcdset{
  cells={font=\everymath\expandafter{\the\everymath\displaystyle}},
}

\allowdisplaybreaks

\author[R. Gambheera]{Rusiru Gambheera}
\address{Department of Mathematics, University of California Santa Barbara, CA 93106-3080, USA}
\email{rusiru@ucsb.edu}



\begin{document}

\begin{abstract} For an abelian, CM extension $H/F$ of a totally real number field $F$, we improve upon the reformulation of the Equivariant Tamagawa Number Conjecture for the Artin motive $h_{H/F}$ by Atsuta-Kataoka in \cite{Atsuta-Kataoka-ETNC} and extend the results proved in \cite{Bullach-Burns-Daoud-Seo}, \cite{Dasgupta-Kakde-Silliman-ETNC}, \cite{gambheera-popescu} and \cite{Dasgupta-Kakde} on conjectures by Burns-Kurihara-Sano \cite{Burns-Kurihara-Sano} and Kurihara \cite{Kurihara}. Then, we consider the $\mathbb{Z}_p[[Gal(H_\infty/F)]]-$module $X_S^{T}$ where $p>2$ is a prime and $H_{\infty}$ is the cyclotomic $\mathbb{Z}_p-$ extension of $H$. This is a generalization of the classical unramified Iwasawa module $X$. By taking the projective limits of the results proved at finite layers of the Iwasawa tower, as our main result, extending the earlier results of Gambheera-Popescu in \cite{gampheera-popescu-RW}, we calculate the Fitting ideal of $X_S^{T,-}$ for non-empty $T$, which is an integral equivariant refinement of the Iwasawa main conjecture for totally real fields proved by Wiles. We also give a conjectural answer to the Fitting ideal of the module $X^-$. 
\end{abstract}

\title[An Integral Equivariant Refinement of the Main Conjecture]{An Integral Equivariant Refinement of the Iwasawa Main Conjecture for Totally Real Fields}

\maketitle

\section{Introduction}
Let $p$ be an odd prime. Let $H/F$ be an abelian CM extension of a totally real field $F$ and let $T$ be a nonempty set of nonarchemedian places in $F$ that are away from the primes above $p$. Now, consider the cyclotomic $\mathbb{Z}_p-$extension of $H$, denoted by $H_{\infty}$, whose $n-$th layer is $H_n$. Let $\mathcal{G}=Gal(H_{\infty}/F)$ and $\mathbb{Z}_p[[\mathcal{G}]]$ be the Iwasawa algebra of $\mathcal{G}$. Now, if $A^T(H_n)$ is the Sylow $p-$subgroup of the $T-$ray class group of $H_n$, we define the following equivariant Iwasawa module ($\mathbb{Z}_p[[\mathcal{G}]]-$module), 
$$X^T:=\varprojlim_n \text{  } A^
T(H_n \text{  })$$
where the transition maps are induced by the norm maps between layers. Notice that $X^T$ is a variant of the classical unramified Iwasawa module, $X$. If $T=\emptyset$, then $X^T$ is precisely $X$.\\

Now, the Iwasawa main conjecture for totally real fields proved by Wiles (Theorem 1.2 in \cite{Wiles}) can be reformulated (see Theorem \ref{Wiles}) as follows. 

\begin{theorem}\label{wiles-intro}\textbf{(Wiles)}
    Under the simplifying assumptions that $H$ contains $p-$th roots of unity and the cyclotomic $\mathbb{Z}_p-$extension of $F_{\infty}/F$ of $F$ and $H/F$ are disjoint, we have the following equality of ideals.
    $$Fitt_{\mathbb{Q}_p(\mathbb{Z}_p[[\mathcal{G}]]^-)}(X^{T,-}\otimes_{\mathbb{Z}_p}\mathbb{Q}_p)=(\Theta_{S_p\cup S_{\infty}}^T(H_{\infty}))$$
\end{theorem}
Here $\ast^-$ denotes the $(-1)$--eigenspace of $\ast$ under the action of the unique complex conjugation automorphism in $\mathcal{G}$ and $\mathbb{Q}_p(\mathbb{Z}_p[[\mathcal{G}]]^-):=\mathbb{Z}_p[[\mathcal{G}]]^-\otimes_{\mathbb{Z}_p}\mathbb{Q}_p$. $Fitt$ denotes the $0-$th Fitting ideal and $\Theta_{S_p\cup S_{\infty}}^T(H_{\infty})$ is the equivariant $p-$adic $L-$function attached to the data $(H_{\infty},F,T,p)$. Check section 5 for the precise definitions.\\

As explained in Remark \ref{not-equivariant}, the theorem above is a ``character by character" result rather than an equivariant result. Moreover, since the module $X^{T,-}$ is tensored by $\mathbb{Q}_p$, there is a significant loss of integral information. To resolve both of these issues one might consider computing $Fitt_{\mathbb{Z}_p[[\mathcal{G}]]^-}X^{T,-}$ in terms of an equivariant $p-$adic L-function. But, the module $X^{T,-}$  is only very rarely of finite projective dimension. As a result its Fitting ideal is not principal and was considered difficult to compute directly. Therefore, equivariant integral refinements of the Iwasawa main conjecture that appear in the literature usually link a module $M$ that is cohomologically better behaved and have projective dimension 1 (hence, the Fitting ideal of $M$ is principal) to integral equivariant $p-$adic L functions. These modules are closely related to the classical Iwasawa modules $X$.\\

For instance, in the equivariant main conjecture of Greither-Popescu \cite{Greither-Popescu}, which was proved under the $\mu=0$ conjecture, the module $M$ is $T_p(\mathcal{M}^{H_{\infty}}_{S,T})$, the $p-$adic realization of the abstract 1-motive, $\mathcal{M}^{H_{\infty}}_{S,T}$ attached to the data $(H_{\infty},F,S,T,p)$ where $S$ is an auxiliary set of places of $F$ that is disjoint from $T$. $T_p(\mathcal{M}^{H_{\infty}}_{S,T})$ is closely related to the classical Iwasasa module, $\mathfrak{X}_S$, the Galois group of the maximal pro-$p$ abelian extension of $H_{\infty}$, unramified away from $S\cup S_p$ where $S_p$ is the set of primes above $p$ in $F$. Using the work of Dasguta-Kakde \cite{Dasgupta-Kakde} on the Brumer-Stark conjecture, Gambheera-Popescu \cite{gambheera-popescu} reformulated and proved the main conjecture of Greither-Popescu without the $\mu=0$ condition. In their reformulation, the role of $M$ is played by $Sel_S^T(H_{\infty})_p$, a certain Selmer module defined at the top of the cyclotomic Iwasawa tower. This module is an extension of the module $\varprojlim_n\text{ } A^T(H_n)^{\vee}$ where $A^T(H_n)^{\vee}=Hom_{\mathbb{Z}_p} (A^T(H_n),\mathbb{Q}_p/\mathbb{Z}_p)$, the Pontragin dual of $A^T(H_n)$ and the transition maps are induced by the inclusion maps. In \cite{gampheera-popescu-RW}, Gambheera-Popescu proved an equivariant main conjecture for the Ritter-Weiss module, $\nabla_S^T(H_{\infty})_p$ at $H_{\infty}$. This module contains $X_S^T:=\varprojlim_n \text{ } A_S^T(H_n)$ where $A_S^T(H_n)$ is the Sylow $p-$subgroup of a certain $(S,T)-$ray class group $Cl_S^T(H_n)$ (see section 3 for the definition). Consequently $Fitt_{\mathbb{Z}_p[[\mathcal{G}]]^-}(X_S^{T,-})$ was calculated under some conditions on $S$ and $T$.\\

As the main result of this paper (Theorem \ref{main theorem}), we calculate $Fitt_{\mathbb{Z}_p[[\mathcal{G}]]^-}(X_S^{T,-})$ for $S, T$ generalizing Theorem 4.9 of \cite{gampheera-popescu-RW}, directly without going through an equivariant main conjecture for a ``nice" Iwasawa module $M$.

\begin{theorem}
    Suppose that $S$ and $T$ satisfy Assumption \ref{assumption-Iwasawa}. Then, we have the following.
    $$Fitt_{\mathbb{Z}_p[[\mathcal{G}]]^-}(X_S^{T,-})=\Theta_{((S\cap S_{ram})\cup S_p)}^T(H_{\infty})\prod_{v\in S_{ram}\setminus (S\cup T\cup S_p)} Q_v \prod_{v\in S_p\setminus S}\Tilde{J}_v^{\infty}\prod_{v\in S\cap S_{ram}}R_v$$
    Here $S_{ram}=S_{ram}(H_{\infty}/F)$.
\end{theorem}

The first term on the right hand side is a equivariant $p-$adic L-function. The other terms are essentially defined in terms of the splitting and ramification data at the primes $v$. For precise definitions see section 5. Observe that the Assumption \ref{assumption-Iwasawa} is much relaxed than the conditions for $S,T$ in Theorem 4.9 in \cite{gampheera-popescu-RW}. For instance, now we have the option of putting $S=S_{\infty}$ which gives us the following.

\begin{corollary}
    Suppose $T\cap S_p=\emptyset$. Then, we have the following.
    $$Fitt_{\mathbb{Z}_p[[\mathcal{G}]]^-}(X^{T,-})=\Theta_{S_p \cup S_{\infty}}^T(H_{\infty})\prod_{v\in S_{ram}\setminus (T\cup S_p)} Q_v \prod_{v\in S_p}\Tilde{J}_v^{\infty}$$
\end{corollary}
Observe that this is an integral equivariant refinement of the Wiles' main conjecture for totally real fields, Theorem \ref{wiles-intro}. By ``tensoring" the above result by $\mathbb{Q}_p$, one can recover Theorem \ref{wiles-intro}.\\

Notice that we still do not have the option of putting $T=\emptyset$, as $T$ is always assumed to be nonempty. It is important to emphasize that this is more than a technical condition arising from our strategy, because if we plug in $T=\emptyset$ to the right hand side, we might get a non-integral element. However, based on the Corrollary above, we give a conjectural answer to $Fitt_{\mathbb{Z}_p[[\mathcal{G}]]^-}(X^{-})$. See Conjecture \ref{conjecture}. \\

The strategy to prove the main theorem of this paper is as follows.\\

For an abelian CM extension $H/F$ of a totally real field $F$, the minus part of the Equivariant Tamagawa Number Conjecture (ETNC) for the Artin motive $h_{H/F}$ was proved away from the prime 2, by Bullach-Burns-Daoud-Seo in \cite{Bullach-Burns-Daoud-Seo} and at all primes by Dasgupta-Kakde-Silliman in \cite{Dasgupta-Kakde-Silliman-ETNC}. Both the proofs took an Euler systems approach. In \cite{gampheera-popescu-RW}, Gambheera-Popescu also gave an alternative proof, away from the prime 2, using their main conjecture for the Ritter-Weiss module and the method of Taylor-Wiles primes. In \cite{Atsuta-Kataoka-ETNC}, Kataoka-Atsuta give a reformulation of ETNC, in terms of their variant of the Ritter-Weiss module, $\Omega_{S'}^T(H)$ attached to the data $(H/F,S',T)$ where $S'$ and $T$ are two sets of places in $F$ satisfying some conditions. Namely, their reformulation computes $Fitt_{\mathbb{Z}[G]^-}(\Omega_{S'}^T(H))$ in terms of an equivariant L-function where $G=Gal(H/F)$. We define a new module $\Omega_{S,S'}^T(H)$ for the disjoint auxiliary sets $S,S'$ and $T$ which simultaneously generalizing both the Ritter-Weiss module, $\nabla_S^T(H)^-$ considered in \cite{Dasgupta-Kakde} and $\Omega_{S'}^T(H)$ considered in \cite{Atsuta-Kataoka-ETNC}. See Remark \ref{Omega-equiv}. This module lies in the following short exact sequence.

\begin{equation}
0\xrightarrow[]{} Cl_S^T(H)^-\xrightarrow[]{}\Omega_{S,S'}^T(H)\xrightarrow[]{} Y_{S}^- \oplus Z_{S'}^-\xrightarrow[]{} 0
\end{equation}
The modules $Y_{S}$ and $Z_{S'}$ are defined in section 3.
 Now, using ETNC we can calculate $Fitt_{\mathbb{Z}_p[G]^-}(\Omega_{S,S'}^T(H)\otimes \mathbb{Z}_p)$. By doing this calculation at each layer $H_n$ of the cyclotomic Iwasawa tower, we obtain partial information about $Fitt_{\mathbb{Z}_p[G]^-}(A_S^{T,-}(H_n))$. This is done in Theorem \ref{equality}. For this we use the  the Shifted Fitting ideal calculations done in \cite{Atsuta-Kataoka-Fitt} and \cite{gampheera-popescu-RW}. Now, by taking projective limits of these results, using techniques developed essestially by Greither-Kurihara in \cite{Greither-Kurihara} we obtain the main result.\\
 
 Moreover, essentially due to Remark \ref{Omega-equiv} and the reformulation of ETNC, we also get the following generalization (Corollary \ref{global-Burns-Kurihara-Sano}) of Theorem 1.14 in \cite{gambheera-popescu}, which is a conjecture by Burns-Kurihara-Sano \cite{Burns-Kurihara-Sano}.

\begin{theorem}\label{intro-global-Burns-Kurihara-Sano}
    Suppose that $S_{ram} \subseteq S \cup T$, $\mu(H)^T$ is trivial and $T$ does not contain any wildly ramified primes away from 2. Then, we have the following.

    $$Fitt_{\mathbb{Z}[G]^-}(\nabla_S^T(H)^-)=(\Theta_S^T)$$
\end{theorem}
Notice that here the conditions on $S,T$ are relaxed as now we allow $T$ to contain tamely ramified primes which was not an option in Theorem 1.14(2) of \cite{gambheera-popescu}. As a consequence, we also obtain the following extension (Corollary \ref{global-strong-Kurihara}) of the Kurihara's conjecture  proved in \cite{Dasgupta-Kakde} (check Theorem 3.5 of loc.cit).
\begin{theorem}\label{intro-global-strong-Kurihara}
    Suppose $\mu(H)^T$ is trivial and $T$ does not contain any wildly ramified primes away from 2. Then we have the following equality of ideals in $\mathbb{Z}[G]^-$. 
    $$Fitt_{\mathbb{Z}[G]^-}(Cl^{T}(H)^{\vee,-})= (\Theta_{S_{\infty}}^T)\prod_{v\in S_{ram}(H/F)\setminus T} (N(I_v), 1-e_{I_v}\sigma_v^{-1})$$
\end{theorem}
Again, observe that, unlike the original formulation of Kurihara's conjecture, now we are allowing $T$ to contain tamely ramified primes. We also remark that in the $p-$adic versions of the above two theorems (Theorem \ref{burns-kurihara-sano} and Theorem \ref{strong-kurihara}), the conditions on $S$ and $T$ are even more relaxed.\\

Organization of this paper is as follows. In section 2 we gather some algebraic tools which are needed to take projective limits of the Fitting ideals over the Iwasawa tower. We also record some results on the shifted Fitting ideals of the module $Z_{S'}$ above. In section 3, we work towards constructing and proving some basic results on the module $\Omega_{S,S'}^T(H)$ improving upon the work in \cite{Atsuta-Kataoka-ETNC}. In section 4 we use the ETNC and its reformulation to prove Theorem \ref{intro-global-Burns-Kurihara-Sano} and Theorem \ref{intro-global-strong-Kurihara}. In section 5 we take the projective limits of the results proved in section 4 to obtain our main theorem.

\subsection*{Acknowledgments} We would like to thank Cristian Popescu for several stimulating discussions. We would also like to thank Francesc Castella for bringing the work of Kataoka and Atsuta to our attention.

\section{Algebraic Preliminaries}

In this section we are establishing some general facts about projective systems of compact commutative rings. More specifically, we are computing the projective limits of their fractional ideals. 

\begin{remark}
Through out this section we assume that the rings  we are dealing with are commutative rings with identity, of characteristic 0 and no integer is a zero-divisor. All the rings that we see in later sections are going to satisfy these conditions.
\end{remark}

\begin{definition}
    Suppose $(R_n)_n$ is a sequence of commutative rings and $I_n$ be a fractional ideal of $R_n$ for each $n\in \mathbb{N}$. Then, we say that the sequence $(I_n)_n$ has bounded denominators if there exist a nonzero integer $N$ such that $N\cdot I_n \subseteq R_n$ for each $n\in \mathbb{N}$.
\end{definition}

\begin{lemma} \label{proj-ideal- lemma-1}
Suppose that for each $n\in\mathbb{N}$,  $R_n$ is a compact commutative  ring. Moreover, suppose that for each $n\in \mathbb{N}$, let $\{I_j^{(n)}\}_{j=1}^k$ be a family of fractional ideals of $R_n$. Assume that,

\begin{enumerate}
\item for each $j$, the sequence $(I_j^{(n)})_n$ has bounded denominators,
    \item for each $m, n\in\mathbb{N}, m>n$ there exist surjective transition maps $\pi_n^{m}:R_{m}\xrightarrow[]{} R_n$ such that $\pi_n^{n+1}(I_j^{(n+1)})=I_j^{(n)}$ for each $j$,
    \item for each $j$ the fractional ideal $I_j^{(\infty)}:=\varprojlim_n I_j^{(n)}$ of $R_{\infty}:=\varprojlim_n R_n$ is finitely generated.
\end{enumerate}
Then we have the following equality of fractional ideals of $R_{\infty}$.

$$\varprojlim_n \prod_{j=1}^k I_j^{(n)}= \prod_{j=1}^k I_j^{(\infty)}$$
\end{lemma}
\begin{proof} Let $N\in \mathbb{N}$ be chosen such that $N\cdot I_j^{(n)}\subseteq R_n$ for each $n$ and $j$. Since $N$ is invertible in each ring $R_n$, we can extend the transition maps to $\pi_n^{n+1}:R_{n+1}[1/N]\xrightarrow[]{} R_n[1/N]$. Since for each $n$ and $j$ we have $I_j^{(n)}\subseteq R_n[1/N]$, the second condition makes sense. Moreover we have $N\cdot I_j^{(\infty)}\subseteq R_{\infty}$ and $N$ is invertible in $R_{\infty}$. Hence, $I_j^{(\infty)}$ considered in the third condition is indeed a fractional ideal of $R_{\infty}$.\\

 Let $J_n=\prod_{j=1}^k I_j^{(n)}$ and $J_{\infty}=\varprojlim_n J_n$. It is easy to see that $\prod_{j=1}^k I_j^{(\infty)} \subseteq J_{\infty}$. For the other inclusion, let $x=(x_n)_n\in J_{\infty}$. Now, let $\pi_n:R_{\infty}\xrightarrow[]{} R_n$ be the obvious projection map. Define $V_n:=\pi_n^{-1}(x_n)\cap \prod_{j=1}^k I_j^{(\infty)}$. Now, by the second condition, $\pi_n(\prod_{j=1}^k I_j^{(\infty)})=J_n$ and hence, $V_n$ is nonempty. On the other hand, continuity of $\pi_n$ and the fact that each $I_j^{(\infty)}$ is finitely generated implies that $V_n$ is a closed subset of $R_{\infty}$. Moreover, observe that for each $n$ we have $V_{n+1}\subseteq V_n$. Now, as each $R_n$ is compact, so is $R_{\infty}$. This implies that $V:=\cap_{n=1}^{\infty} V_n$ is nonempty. So, if $y\in V$, by the definition of $V$, we have $\pi_n(y)=x_n$. Hence, we have $x=y\in V\subseteq \prod_{j=1}^k I_j^{(\infty)}$ as desired.   
\end{proof}
Now we prove the following variation of the previous lemma for two sequences of fractional ideals with slightly different conditions on the fractional ideals and the rings. 

\begin{lemma}\label{Proj-ideal-lemma}
  Suppose $(R_n)_n$ is a system of compact Noetherian rings with surjective transition maps $\pi_n^{m}:R_{m}\xrightarrow[]{} R_n$ when $m>n$. Let $(I_n)_n$ and $(J_n)_n$ be sequences of fractional ideals of $R_n$ with bounded denominators where $\pi_n^{n+1}(J_{n+1})\subseteq J_n$ for each $n\in\mathbb{N}$. Moreover, assume that $I_n=(x_1^{(n)}, x_2^{(n)}, \text{ ... } , x_k^{(n)})$ and $\pi_n^{n+1}(x_i^{(n+1)})=x_i^{(n)}$ for each $n$ and $i$. Then we have the following equality of fractional ideals of $R_{\infty}:=\varprojlim_n R_n$.

  $$\varprojlim_n (I_n J_n)=(\varprojlim_n I_n)(\varprojlim_n J_n)$$
\end{lemma}
\begin{proof}
    As explained in the beginning of the proof of previous lemma, since the given sequences of fractional ideals have bounded denominators, via extended transition maps, the given conditions are well defined.\\
    
    It is obvious that $(\varprojlim_n I_n)(\varprojlim_n J_n) \subseteq \varprojlim_n (I_n J_n)$. Now, let $y=(y_n)_n \in \varprojlim_n (I_n J_n)$. Then, for each $n\in\mathbb{N}$, $y_n=\sum_{i=1}^k x_i^{(n)} u_i^{(n)}$ for some $k\in\mathbb{N}$, where each $u_i^{(n)} \in J_n$. Now, we define, new elements $\Tilde{u}_i^{(n)}\in J_n$ in the following way. First let us do this task for $n=1$. Consider the sequence $$(z_n^{(1)})_n=((\pi_1^n(u_1^{(n)}),\pi_1^n(u_2^{(n)}), \text{ ... } , \pi_1^n(u_k^{(n)})))_n$$ in $I_1^k$. Since, $R_1$ is Noetherian and compact, $J_1$ is finitely generated and hence, compact. So, the above sequence has a convergent sub sequence, say $(z_{n_j}^{(1)})_j$. Now, for each $i$ define $$\Tilde{u}_i^{(1)}:=\lim_{j\to \infty} \pi_1^{n_j}(u_i^{(n_j)}).$$ Observe that for all $j$, $y_1=\pi_1^{n_j}(y_{n_j})$. Therefore, 
    $$y_1=\lim_{j\to\infty} \pi_1^{n_j}(y_{n_j})=\sum_{i=1}^k \lim_{j\to\infty} \pi_1^{n_j}(x_i^{(n_j)}u_i^{(n_j)})=\sum_{i=1}^k x_i^{(1)}\Tilde{u}_i^{(1)}$$
    Now, for each $i$, let us define $\Tilde{u}_i^{(2)}$. Consider the following sequence,
    $$(z_{n_j}^{(2)})_j=((\pi_2^{n_j}(u_1^{({n_j})}),\pi_2^{n_j}(u_2^{({n_j})}), \text{ ... } , \pi_2^{n_j}(u_k^{({n_j})})))_{n_j}$$ in $I_2^k$. As before, by the compactness of $I_2^k$, we can pick a further subsequence of $(n_j)_j$, say $(n_{j_r})_r$ such that $(z_{n_{j_r}}^{(2)})_r$ is convergent. Now, define $$\Tilde{u}_i^{(2)}:=\lim_{r\to \infty} \pi_2^{n_{j_r}}(u_i^{(n_{j_r})})$$ Following the same steps as before, we can prove that,
    $$y_2=\sum_{i=1}^k x_i^{(2)}\Tilde{u}_i^{(2)}$$
    Furthermore, observe that, $$\pi_1^2(\Tilde{u}_i^{(2)})=\pi_1^2(\lim_{r\to \infty} \pi_2^{n_{j_r}}(u_i^{(n_{j_r})}))=\lim_{r\to \infty} \pi_1^{n_{j_r}}(u_i^{(n_{j_r})})=\Tilde{u}_i^{(1)}$$
    Following the same method, by picking further subsequences of  $(n_{j_r})_r$, for each $i$ one can define the family of sequences, $(\Tilde{u}_i^{(n)})_n$ such that, for each $n$ and $i$, we have $\pi_n^{n+1}(\Tilde{u}_i^{(n+1)})=\Tilde{u}_i^{(n)}$ and $y_n=\sum_{i=1}^k x_i^{(n)}\Tilde{u}_i^{(n)}$. Therefore, $u_i=(\Tilde{u}_i^{(n)})_n \in \varprojlim_n J_n$ and $y=\sum_{i=1}^k x_i u_i$ where $x_i=(x_i^{(n)})_n\in \varprojlim_n I_n$. Hence, $y\in (\varprojlim_n I_n)(\varprojlim_n J_n)$ as desired. This completes the proof. 
\end{proof}

The next theorem we prove will be applied in later sections when we take projective limits of fractional ideals of rings endowed with the $p-$adic topology.

\begin{lemma}\label{ignore-generators}
  Let $p$ be a prime. Suppose $(R_n)_n$ is a system of compact Noetherian rings with surjective transition maps $\pi_n^{m}:R_{m}\xrightarrow[]{} R_n$. Let $(I_n)_n$ and $(J_n)_n$ be sequences of fractional ideals of $R_n$ with bounded denominators such that $\pi_n^{n+1}(I_{n+1})\subseteq I_n$ and $\pi_n^{n+1}(J_{n+1})\subseteq J_n$. Assume that for each $n\in\mathbb{N}$, $$I_n=(x_1^{(n)}, x_2^{(n)}, \text{ ... } , x_k^{(n)} , y_1^{(n)}, y_2^{(n)}, \text{ ... } , y_l^{(n)})$$ where $\pi_n^{n+1}(x_i^{(n+1)})=x_i^{(n)}$ and $\pi_n^{n+1}(y_i^{(n+1)})=p\cdot y_i^{(n)}$for each $i$. Let us also assume that $\lim_{k\to\infty} p^k=0$ in each $R_n$. Then we have the following equality of fractional ideals of $R_{\infty}:=\varprojlim_n R_n$. 
  $$\varprojlim_n (I_n J_n)=\varprojlim_n (\Tilde{I}_n J_n)$$
  where $\Tilde{I}_n=(x_1^{(n)}, x_2^{(n)}, \text{ ... } , x_k^{(n)} )$.
\end{lemma}
\begin{proof}
   Since $\Tilde{I}_n \subseteq I_n$ for each $n$, we have $\varprojlim_n (\Tilde{I}_n J_n) \subseteq \varprojlim_n (I_n J_n)$. Now, let $y=(y_n)_n \in \varprojlim_n (I_n J_n)$. Then, for each $m$, we have $$y_m=\sum_{i=1}^k x_i^{(m)}u_i^{(m)}+\sum_{i=1}^l y_i^{(m)}v_i^{(m)}$$ where $u_i^{(m)}, v_i^{(m)}\in J_m$ for each $i$ and $m$. Now, for each $m>n$, we have $$y_n=\pi_n^m(y_m)=\sum_{i=1}^k x_i^{(n)}\pi_n^m(u_i^{(m)})+\sum_{i=1}^l p^{m-n} y_i^{(n)}\pi_n^m (v_i^{(m)})$$
Now, for a fixed $n$ consider the sequence,
$$(z_m)_{m>n}=(\pi_n^m(u_i^{(m)}))_{i=1}^k$$
in the compact space $I_n^k$. We can choose a convergent subsequence, say $(z_{m_j})_j$. Let $\Tilde{u}_i^{(n)}=\lim_{j\to \infty} \pi_n^{m_j}(u_i^{(m_j)})\in J_n$. Now, observe that $$\lim_{j\to \infty} p^{m_j-n} y_i^{(n)}\pi_n^{m_j} (v_i^{(m_j)})=0$$ Therefore, we have that, for each $n$,
$$y_n=\lim_{j\to \infty} \pi_n^{m_j} (y_{m_j})=\sum_{i=1}^k x_i^{(n)} \Tilde{u}_i^{(n)}\in \Tilde{I}_n J_n$$ Therefore, $y\in \varprojlim_n (\Tilde{I}_n J_n)$ as desired.
\end{proof}

Now, we recall some of the results of Kataoka and Atsuta on shifted Fitting ideals of some modules over a group ring. For basic definitions of Fitting ideals and their shifts, see chapter 2 of \cite{Atsuta-Kataoka-Fitt}.\\

Let $G,D$ and $I$ be finite abelian groups such that $I\subseteq D\subseteq G$ and $D/I$ is cyclic. Let us consider the $\mathbb{Z}[G]$-module $A=\mathbb{Z}[G]/(i-1,1-\sigma^{-1}+|I|; i\in I)$ where $\sigma\in D$ is a lift of a fixed generator of the group $D/I$. Observe that the module $A$ does not depend on this lift. In \cite{Atsuta-Kataoka-Fitt} the authors computed the first shifted Fitting ideal $Fitt_{\mathbb{Z}[G]}^{[1]}(A)$. Let us start by recalling some definitions that are needed to state their result.

\begin{definition}
    For any finite group $H$, define the following elements of the group ring $\mathbb{Q}[H]$.
    $$N(H):=\sum_{h\in H}h$$
    $$e_H:=\frac{1}{|H|}N(H)$$
\end{definition}

\begin{definition}
    Let $I=I_1\times I_2 \times \text{ ... } I_s$ where each $I_j=\langle i_j \rangle$ is cyclic. Then, define the following ideals of $\mathbb{Z}[G]$. 
    $$Z_i:=\left(\prod_{j=1}^{s-i} N(I_{l_j}); 1\leq l_1 < l_2 < \text{ ... } , < l_{s-i}\leq s \right)$$
     $$\Delta D:=Ker(\mathbb{Z}[G]\xrightarrow[]{} \mathbb{Z}[G/D])$$
Moreover we define the following ideal.
$$\mathcal{J}(I):=\sum_{i=1}^s Z_i \Delta D^{i-1} $$
\end{definition}
In \cite{Atsuta-Kataoka-Fitt} the authors proved that $\mathcal{J}(I)$ is independent on the cyclic decomposition of $I$ (check the chapter 4 of \cite{Atsuta-Kataoka-Fitt}), which justifies the notation.\\

For the applications of the next chapters we would like to keep track of the generators of $\mathcal{J}(I)$. Observe that,
$$\Delta D=(\sigma-1, i-1; i\in I)=(\sigma-1, i_j-1; 1\leq j \leq s).$$

Therefore, it is easily seen that we can write $\mathcal{J}(I)$ in terms of its generators in the following way.

\begin{lemma} \label{Shifted-Fitt-generators}
    Consider the following sets of tuples of positive integers.
    $$L_i=\{(l_1,l_2, \text{ ... } l_{s-i}); 1\leq l_1 < l_2 < \text{ ... } , < l_{s-i}\leq s\}$$
    $$M_j=\{(m_1,m_2, \text{ ... } m_{j}); 1\leq m_1 \leq m_2 \leq \text{ ... } , \leq m_{j}\leq s\}$$
For, tuples $\lambda=(l_1,l_2, \text{ ... } l_{s-i})\in L_i$ and $\mu=(m_1,m_2, \text{ ... } m_{j})\in M_j$, define the element $e(\lambda,\mu)\in \mathbb{Z}[G]$,
$$e(\lambda,\mu):=\prod_{j=1}^{s-i} N(I_{l_j})\prod_{r=1}^{j}(i_{m_r}-1)$$
Then, we have the following.
$$J=(e(\lambda,\mu)(\sigma-1)^{i-1-j}; \lambda\in L_i , \mu\in M_j , j\leq i-1<s)$$
\end{lemma}

The following is the Theorem 4.1 in \cite{Atsuta-Kataoka-Fitt}.
\begin{theorem}\label{A-shifted-Fitt}
   Let $h=1-e_I\sigma^{-1}+N(I)\in \mathbb{Q}[G]$. Then, we have the following equality of fractional ideals of $\mathbb{Z}[G]$.
    $$(h)Fitt_{\mathbb{Z}[G]}^{[1]}(A)=(N(I),(1-e_I\sigma^{-1})\mathcal{J}(I))$$
\end{theorem}
It can be easily checked that $h\in \mathbb{Q}[G]$ is a non-zero divisor. Therefore, the above theorem determines $Fitt_{\mathbb{Z}[G]}^{[1]}(A)$.\\

In the rest of this section we introduce some notations to be used in later chapters.  If $G$ is a (pro)finite abelian group and $M$ is a $\Bbb Z[G]$--module, we endow its Pontrjagin dual
$$M^\vee:={\rm Hom}_{\Bbb Z}(M, \Bbb Q/\Bbb Z)$$
with the {\it covariant} $G$--actions $gf(x)=f(gx)$, for $g\in G$. Further, if $G$ contains a canonical element $c$ of order $2$ (usually a complex conjugation automorphism), then we let
$$M^-:=\frac{1}{2}(1-c)\cdot (M\otimes_{\Bbb Z}\Bbb Z[1/2]).$$
This is the $(-1)$--eigenspace of $M\otimes_{\Bbb Z}\Bbb Z[1/2]$ under the action of $c$. It has a natural structure of $\Bbb Z[1/2][G]^-$--module. If $p$ is an odd prime, then we let
$$M_p:=M\otimes_{\Bbb Z}\Bbb Z_p, \qquad M_p^-:=\frac{1}{2}(1-c)M_p=M^-\otimes_{\Bbb Z[1/2]}\Bbb Z_p,$$
viewed as modules over $\Bbb Z_p[G]$ and $\Bbb Z_p[G]^-$, respectively. Obviously, $(\ast\to \ast^-)$ and $(\ast\to \ast^-_p)$ are exact functors. Tacitly, we identify the rings $\Bbb Z_p[G]/(1+c)$ and $\Bbb Z_p[G]^-$, respectively  $\Bbb Z[1/2][G]/(1+c)$ and $\Bbb Z[1/2][G]^-$, via the isomorphism given by multiplication with $\frac{1}{2}(1-c)$.If $N$ is a $\Bbb Z_p[[G]]$--module, we also use the same notations as above
$$N^\vee:={\rm Hom}_{\Bbb Z_p}(N, \Bbb Q_p/\Bbb Z_p)$$
to denote its Pontryagin dual, always endowed with the covariant $G$--action.

\section{The modified Ritter-Weiss module }

In this chapter we construct a Galois module $\Omega_{S,S'}^T(H)$ that is a slight modification of the Ritter-Weiss modules considered in \cite{Dasgupta-Kakde} and its variant in \cite{Atsuta-Kataoka-ETNC}. The original construction of the module goes back to the work of Ritter and Weiss \cite{Ritter-Weiss}. It has been modified and generalized over the years by series of work due to Greither \cite{Greither-ETNC}, Dasgupta-Kakde \cite{Dasgupta-Kakde}, Kurihara \cite{Kurihara} and Atsuta-Kataoka \cite{Atsuta-Kataoka-ETNC}. Our construction recovers the modules in \cite{Dasgupta-Kakde} and \cite{Atsuta-Kataoka-ETNC} as special cases. \\

First, we recall some notations and definitions from \cite{Dasgupta-Kakde} and \cite{Atsuta-Kataoka-ETNC}. Let $H/F$ be an abelian extension of number fields with Galois group $G$ where $H$ is CM and $F$ is totally real. Moreover, let $S_\infty$ be the set of archimedean places in $F$. We fix three disjoint, finite sets of places $S, S'$ and $T$ of $F$, such that $S_\infty\subseteq S$. We denote the sets of places in $H$ above the places in $S, S', S_{\infty}$ and $T$ by $S_H, S_H', S_{\infty}(H)$ and $T_H$ respectively. When there is no risk of confusion we may drop $H$ from the notation. We define
$$\mathcal O_{H,S,T}^{\times}:= \{x\in H^\times ; {\rm ord}_w(x)= 0, \text{ for all }w\not\in S, \quad {\rm ord}_{w}(x-1)>0, \text{ for all } w\in T\},$$
where for a finite prime $w$ in $H$, we let ${\rm ord}_w$ be the normalized valuation of $H$ associated to $w$.

Moreover, for any two sets of places $\Sigma$ and $\Sigma'$ of $F$, $Cl^{\Sigma'}(H)$ denotes the ray--class group of $H$ of conductor $(\prod_{w\in \Sigma'_H}w)$ and $Cl^{\Sigma'}_{\Sigma}(H)$ denotes the quotient $Cl^{\Sigma'}(H)$ by the subgroup generated by the classes of ideals in $\Sigma_H \setminus S_\infty(H)$.\\

In what follows, if $w$ is a prime in $H$, then $H_w$ denotes the completion of $H$ in its $w$--adic topology. If $w$ is finite, then $\mathcal O_w$ denotes the ring of integeres in $H_w$, and $U_w$  denotes the subgroup of principal units inside $\mathcal O_w^\times$. Also, $J_H$ and $C_H$ denote the topological groups of id\`eles and id\`ele classes of $H$, respectively. We denote by $G_w$ and $I_w$ the decomposition group and the inertia group of $w$ in $H/F$. When there is no risk of confusion we may denote those groups by $G_v$ and $I_v$ respectively, where $v$ is the prime in $F$ below $w$. We let $\Delta G$ and $\Delta G_w$ the augmentation ideals of $\Bbb Z[G]$ and $\Bbb Z[G_w]$, respectively.
\begin{definition}\label{larger-S'}  
We call a finite set  $S''$ of places in $F$, an $(H/F, S,S', T)$--large set if.
\begin{itemize}
  \item $S\cup S' \subseteq S''$ and $S''\cap T=\phi$
  \item Each element in $S''\setminus (S\cup S')$ is unramified in $H/F$.
  \item  $Cl_{S''}^T(H)=1$
  \item $\bigcup_{w\in S''(H)} G_w=G$. 
  \item There exist an element $v_0\in S''\setminus (S\cup S')$ such that $G_{v_0}=Gal(H/H^+)$ where $H^+$ is the maximal totally real subfield of $H$. 
\end{itemize}
\end{definition}
\medskip

\noindent For the rest of this section, we fix an $(H/F, S,S', T)$--large
 set $S''$. The existence of such a set follows from a standard Chebotarev density theorem argument.\\

Let us  fix a place $v$ of $F$ and a place $w$ of $H$ which sits above $v$. Following Ritter and Weiss \cite{Ritter-Weiss}, we define a $G_w$-module $V_w$, as an extension of $\Delta G_w$ by $H_w^\times$
\begin{equation}\label{local-exact-sequence}
0\xrightarrow[]{} H_w^{\times}\xrightarrow[]{s}V_w\xrightarrow[]{}\Delta G_w\xrightarrow[]{}0, \end{equation}
whose extension class $\alpha_w$ corresponds via the canonical isomorphisms (see \cite{Ritter-Weiss})
$$Ext_{G_w}^1(\Delta G_w,H_w^{\times})=H^1(G_w,Hom(\Delta G_w,H_w^{\times}))\overset{\delta_w'}{\simeq}  H^2(G_w,H_w^{\times}), \qquad \alpha_w\mapsto u_{H_w/F_v}, $$ 
to the local fundamental class $u_{H_w/F_v}\in  H^2(G_w,H_w^{\times})$. 
Further, the composite injective map $\mathcal O_w^\times \subseteq H_w^\times\overset{s}{\longrightarrow} V_w$ gives rise to an exact sequence of $G_w$--modules 
\begin{equation}\label{local-exact-sequence-units}0\xrightarrow[]{} O_w^{\times}\xrightarrow[]{}V_w \xrightarrow[]{}W_w\xrightarrow[]{}0,\end{equation}
where $W_w\simeq V_w/s(\mathcal O_w^\times)$ . 
Moreover, since, $H_w^{\times}/O_w^{\times}\simeq \Bbb Z$ (with $G_w$ acting trivially on $\Bbb Z$),  we have the following exact sequence of $G_w$--modules.
\begin{equation}\label{local-exact-sequence-W}0\xrightarrow[]{}\mathbb{Z}\xrightarrow[]{i_w}W_w\xrightarrow[]{j_w}\Delta G_w\xrightarrow[]{}0.\end{equation}\\
We also have the following alternative description of $W_w$ as a $\mathbb{Z}[G_w]-$module due to Weiss and Gruenberg \cite{Gruenberg-Weiss}. 

\begin{equation}\label{W-description}
  W_w=\{(x,y)\in \Delta G_w\bigoplus \mathbb{Z}[G_w/I_w]; \Bar{x}=(1-\phi_w^{-1})y\}  
\end{equation}
where $\Bar{x}$ is the image of $x$ in $\mathbb{Z}[G_w/I_w]$ and $\phi_w$ is the Frobenius. Clearly, this is a free $\mathbb{Z}$- module. A $\mathbb{Z}$-basis is given by, $$\Bigl\{w_g=\Bigl(g-1,\sum_{i=1}^{a(g)}\phi_w^i\Bigr) ; g\in G_w\Bigr\}$$ where $a(g)$ defined such that for each $g\in G_w, \Bar{g}=\phi_w^{a(g)}$ and $0<a(g)\leq l_v:=|G_w/I_w|$. Under the notation of the short exact sequence \ref{local-exact-sequence-W}, $i_w(1)=w_1$ and $j_w(w_g)=g-1$. Now, the $G_w$ action on these basis elements is given by $g\cdot w_h=w_{gh}-w_g+a_{g,h}w_1$ for each $g,h\in G_w$ where $a_{g,h}$ is defined by $a(g)+a(h)=a(gh)+l_va_{g,h}$.
\\\\

\begin{proposition}\label{local map - W}
    Under the above description of $W_w$, the followings are true.

    \begin{enumerate}
        \item If $H_w/F_v$ is unramified, we have the isomorphism,
        $$\iota_w: W_w\cong \mathbb{Z}[G_w]$$
        which sends $(x,y)$ to $y$.
        \item For any $w$, we have the short exact sequence,
        $$0\xrightarrow[]{} W_w \xrightarrow[]{f_w} \mathbb{Z}[G] \xrightarrow[]{} \mathbb{Z}[G_w/I_w]/(1-\phi_w^{-1}+|I_w|)\xrightarrow[]{} 0$$
        where $f_w$ sends $(x,y)$ to $x+N(I_w)y$. Here by abuse of notation we also denote a lift of the Frobenius of $w$ to $G_w$ by $\phi_w$.
    \end{enumerate}
\end{proposition}
\begin{proof}
    This is Proposition 2.6 of \cite{Atsuta-Kataoka-ETNC}.
\end{proof}

A similar argument gives rise to a global analogue of the exact sequence \eqref{local-exact-sequence}
\begin{equation}\label{global-exact-sequence}0\xrightarrow[]{}C_H\xrightarrow[]{}D\xrightarrow[]{}\Delta G\xrightarrow[]{}0,\end{equation}
whose extension class $\alpha$ corresponds via the canonical isomorphisms (see \cite{Ritter-Weiss})
$$Ext_G^1(\Delta G,C_H)=H^1(G,Hom(\Delta G,C_H))\overset{\delta'}{\simeq} H^2(G,C_H), \qquad \alpha\mapsto u_{H/F}$$
to the global fundamental class $u_{H/F}\in H^2(G,C_H).$

\begin{definition} Assume that $R$ is a set of finite places in $F$, and that for every $v\in R$ we fix a place $w$ in $H$ above $v$ and a $G_w$--module $M_w$.  Then, we define the following $G$--module
\[\widetilde{\prod_{v\in R}} M_w:=\prod_{v\in R} Ind_{G_w}^G M_w,
\]
where $Ind_{G_w}^G M_w:=M_w\otimes_{\Bbb Z[G_w]}\Bbb Z[G]$ is the usual induced module.
\end{definition}

Now, we can define the following $G$--modules
$$J:=\prod_{v\notin S\cup T}^{\sim} O_w^{\times}\times \prod_{v\in S}^{\sim} H_w^{\times}\times \prod_{v\in T}^{\sim} U_w, \qquad 
J':=\prod_{v\notin S''\cup T}^{\sim} O_w^{\times}\times \prod_{v\in S''}^{\sim} H_w^{\times}\times \prod_{v\in T}^{\sim} U_w$$
$$V:=\prod_{v\notin S''\cup T}^{\sim} O_w^{\times}\times \prod_{v\in S''}^{\sim} V_w\times \prod_{v\in T}^{\sim} U_w$$ 
$$W:=\prod_{v\in S''\setminus S}^{\sim} W_w\times \prod_{v\in S}^{\sim} \Delta G_w, \qquad W':=\prod_{v\in S''}^{\sim} \Delta G_w.$$
If we combine \eqref{local-exact-sequence}--\eqref{local-exact-sequence-W}, we obtain the short exact sequences of $G$--modules
\begin{equation}\label{exact-sequences-J}0\xrightarrow[]{} J \xrightarrow[]{} V\xrightarrow[]{} W \xrightarrow[]{} 0,\qquad 
0\xrightarrow[]{} J' \xrightarrow[]{} V \xrightarrow[]{} W' \xrightarrow[]{} 0.\end{equation}
Now, Theorem 1 in \cite{Ritter-Weiss} gives us the  the following commutative diagram of $G$--modules
\[\begin{tikzcd}
& 0\arrow{d} & 0\arrow{d} & 0\arrow{d} \\%
& O_{H,S,T}^{\times}\arrow{d} & V^{\theta}\arrow{d} & W^{\theta}\arrow{d} \\%
0\arrow{r} & J\arrow{r}\arrow{d}{\theta_J} & V \arrow{r}\arrow{d}{\theta_V} & W \arrow{r}\arrow{d}{\theta_W} & 0 \\%
0\arrow{r} & C_H\arrow{r}\arrow[two heads]{d} & D \arrow{r}\arrow{d} & \Delta G\arrow{r}\arrow{d} & 0 \\%
& Cl_S^T(H) & 0 & 0
\end{tikzcd}
\]
where $\theta_J$ and $\theta_W$ are the obvious maps and the existence and uniqueness of $\theta_V$ is a direct consequence of the compatibility between the local and global fundamental classes. 
The snake lemma gives an exact sequence of $G$--modules
\begin{equation}\label{sequence-units-class-S}
0\xrightarrow[]{} O_{H,S,T}^{\times}\xrightarrow[]{} V^{\theta}\xrightarrow[]{} W^{\theta}\xrightarrow[]{} Cl_S^T(H)\xrightarrow[]{} 0.\end{equation}
An identical line of arguments, involving $J'$ and $W'$ instead of $J$ and $W$, produces the following exact sequence of $G$--modules. 
\begin{equation}\label{sequence-class-units-S'}0\xrightarrow[]{} O_{H,S'',T}^{\times}\xrightarrow[]{} V^{\theta}\xrightarrow[]{} W'^{\theta}\xrightarrow[]{} 0.\end{equation}
Moreover, the definition of the maps $\theta_W$ and $\theta_{W'}$ combined with \eqref{local-exact-sequence-W} gives a short exact sequence of $G$--modules
\begin{equation}\label{sequence-W-W'}0\to Y_{S''\setminus S}\to W^\theta\to W'^{\theta}\to 0.\end{equation}
where $Y_{S''\setminus S}$ denotes the free $\Bbb Z$--module of divisors (with integral coefficients) of $H$, supported at primes in $S''_H\setminus S_H$.\\

Observe that since $H$ is a CM field there is a unique complex conjugation automorphism $c\in G$. Throughout this paper, for any $\Bbb Z[G]-$module $M$, we define,
$$M^-:=\frac{1}{2}(1-c)\cdot (M\otimes \mathbb{Z}[1/2]).$$
Observe that this is a $\Bbb Z[G]^-$--module, 
where 
$$\Bbb Z[G]^-:=\Bbb Z[1/2][G]/(1+c).$$ 
Clearly these definitions can be extended to any subgroup of $G$ containing $c$ ($G_{v_0}$ for instance).
\begin{proposition}\label{W-theta-isomorphism}
    We have the following isomorphism of $\mathbb{Z}[G]^-$ modules.
    $$W^{\theta,-}\cong \left(\prod_{v\in S''\setminus (S\cup \{v_0\})}^{\sim} W_w\times \prod_{v\in S}^{\sim} \Delta G_w\right)^-$$
\end{proposition}
\begin{proof}
    Observe that $\Delta G^-=\mathbb{Z}[G]^-$. Therefore, we have the following short exact sequence. 
    $$0\xrightarrow[]{} W^{\theta,-}\xrightarrow[]{} W^- \xrightarrow[]{\theta_W^-} \mathbb{Z}[G]^- \xrightarrow[]{} 0$$
    Now, for any $w_0|v_0$ we have $G_{w_0}=Gal(H/H^+)=\langle c \rangle$. Hence, by the short exact sequence \ref{local map - W} we have that 
    $$W_{w_0}^-\cong (\Delta G_{w_0})^- \cong (\mathbb{Z}[G_{w_0}])^-$$ via the map $j$. Therefore, we have the induced isomorphism $$Ind_{G_{w_0}}^G W_{w_0}^- \cong Ind_{G_{w_0}}^G (\mathbb{Z}[G_{w_0}])^- \cong \mathbb{Z}[G]^-.$$ It is easily seen that this is the $v_0$ component of the the map $\theta_W^-$. Hence the inverse isomorphism is a section of the map $\theta_W^-$. So, the above short exact sequence splits and induces the desired isomorphism. 
\end{proof}

\begin{definition}
    We define the $\mathbb{Z}[G]^--$module homomorphism
    $$f_{S,S'}:W^{\theta,-}\xrightarrow[]{} \bigoplus_{S''\setminus\{v_0\}}\mathbb{Z}[G]^-$$
    componentwise via the isomorphism in Proposition \ref{W-theta-isomorphism}  as follows.

    \begin{itemize}
        \item For $v\in S$, $f$ is induced by the inclusion $\Delta G_w \subseteq \mathbb{Z}[G_w]$ 
        \item For $v\in S'$, $f$ is induced by the map $f_w$ in the Proposition \ref{local map - W} (2).  
        \item For $v\in S''\setminus (S\cup S'\cup \{v_0\})$, $f$ is induced by the map $\iota_w$ in the Prposition \ref{local map - W} (1).  
    \end{itemize}
\end{definition}

Now we are ready to define the modified Ritter-Weiss module.

\begin{definition}\label{Ritter-Weiss-module-definition}
Define the $\mathbb{Z}[G]^--$ module $\Omega_{S,S'}^T(H)$ associated to the data $(H/F, S,S', T)$ by 
\[\Omega_{S,S'}^T(H):=coker\left(\Psi_{S,S'}^T : V^{\theta,-}\xrightarrow[]{} W^{\theta,-} \xrightarrow[]{f_{S,S'}} \bigoplus_{S''\setminus \{v_0\}}\mathbb{Z}[G]^-\right) \] 
\end{definition}
\begin{remark} \label{Omega-equiv}
    Observe that when $S=S_{\infty}$, the above definition coincides with that of $\Omega_{S'}^T(H)$ in \cite{Atsuta-Kataoka-ETNC}. Moreover, when $S'=\emptyset$ and $S_{ram}(H/F)\subseteq S\cup T$, the above definition coincides with the minus part of the Ritter-Weiss module $\nabla_S^T(H)$ defined in \cite{Dasgupta-Kakde}.
\end{remark}
Denote $S\setminus S_{\infty}$, the set of finite places of $S$ by $S_f$. We prove the following technical lemma which will be used in the next chapter. This is a generalization of Lemma 4.16(1) of \cite{Atsuta-Kataoka-ETNC}.
\begin{lemma}\label{change of maps}
  For a place $v$ in $F$ and a place $w$ in $H$ above $v$ we define $h_v=1-e_{I_w}(\phi^{-1}_w-|I_w|)\in \mathbb{Q}[G]$. We have the following commutative diagram of $\mathbb{Q}[G]^--$modules
\[
\begin{tikzcd}[row sep=10ex, column sep=large]
    V^{\theta,-}\otimes\mathbb{Q} \arrow{r}{\Psi_{S_{\infty},S_f\cup S'}}\arrow{dr}{\Psi_{S,S'}} & \bigoplus_{v\in S''\setminus \{v_0\}}\mathbb{Q}[G]^- \arrow{d} \\%
    & \bigoplus_{v\in S''\setminus\{v_0\}}\mathbb{Q}[G]^-
\end{tikzcd}
\]
where the vertical map is defined as the identity map on $v\in S''\setminus (S\cup  \{v_0\})$ and as multiplication by $(1-e_{I_w}\phi_w^{-1})h_v^{-1}$ on $v\in S_f$.
\end{lemma}
\begin{proof} First of all observe that $h_v\in\mathbb{Q}[G]^{\times}$ because all of its components with respect to the characters of $G$ are nonzero. Hence the vertical map makes sense. The proof of this lemma is very similar to that of Lemma 4.16 (1) of \cite{Atsuta-Kataoka-ETNC}. But, we will elaborate it here for the convenience of the reader.\\

First let us prove that  for each place $v$ of $F$ and $w$ of $H$ above $v$, $W_w\otimes \mathbb{Q}$ is a free $\mathbb{Q}[G_w]-$module of rank $1$ with the generator $(1-e_{I_w}\phi_w^{-1},1)$. Here we are using the description \ref{W-description} for $W_w$. \\

Let $f:\mathbb{Q}[G_w] \xrightarrow[]{} W_w\otimes \mathbb{Q}$ be the $\mathbb{Q}[G_w]-$module morphism given by $f(1)=(1-e_{I_w}\phi_w^{-1},1)$. Since $W_w$ has $\mathbb{Z}-$rank $|G_w|$ (see \ref{local-exact-sequence-W} for instance), both $W_w\otimes \mathbb{Q}$ and $\mathbb{Q}[G_w]$ have the same dimension over $\mathbb{Q}$. Hence, in order to show that $f$ is an isomorphism, it is enough to show the surjectivity. Now, let $(x,y)\in W_w\otimes \mathbb{Q}$. Let $\Tilde{y}\in \mathbb{Q}[G_w]$ be a lift of $y$. Then, inside $\mathbb{Q}[G_w/I_w]$ we have,
$$\Tilde{y}\cdot (1-e_{I_w}\phi_w^{-1})=y\cdot (1-\phi_w^{-1})=\overline{x}$$
Hence, we have $$t:=x-\Tilde{y}\cdot (1-e_{I_w}\phi_w^{-1})\in \Delta I_w.$$ But, clearly $e_{I_w}\in \mathbb{Q}[G_w]$ annihilates $\Delta I_w$. Therefore we have $(1-e_{I_w}\phi_w^{-1})t=t$. Hence, $(\Tilde{y}+t)(1-e_{I_w}\phi_w^{-1})=x$. So we have,
$$(\Tilde{y}+t)\cdot (1-e_{I_w}\phi_w^{-1},1)=(x,y)$$
Therefore, $f$ is surjective and thereby $W_w\otimes \mathbb{Q}$ is a free $\mathbb{Q}[G_w]-$module of rank 1. Now, observe that the maps $\Psi_{S,S'}$ and $\Psi_{S_{\infty},S_f\cup S'}$ differ because the maps $f_{S,S'}$ and $f_{S_{\infty},S_f\cup S'}$ differ. Those maps differ only at $S-$places. Now, observe that, 
$$f_w(1-e_{I_w}\phi_w^{-1},1)=h_v$$
$$j_w(1-e_{I_w}\phi_w^{-1},1)=1-e_{I_w}\phi_w^{-1}$$
where $f_w$ is as in Proposition \ref{local map - W} (2) and $j_w$ is as in the short exact sequence \ref{local-exact-sequence-W}. This together with the definitions of $f_{S,S'}$ and $f_{S_{\infty},S_f\cup S'}$ at $S-$places gives us the following commutative diagram.
\[
\begin{tikzcd}[row sep=6ex, column sep=12ex]
     \left(\prod_{v\in S''\setminus\{v_0\}}^{\sim} W_w\otimes \mathbb{Q} \right)^-\times \left(\prod_{v\in S_{\infty}}^{\sim} \Delta G_w \otimes \mathbb{Q}\right)^- \arrow{r}{f_{S_{\infty},S_f\cup S'}} \arrow{d}        & \bigoplus_{v\in S''\setminus\{v_0\}}\mathbb{Q}[G]^- \arrow{d} \\%
   \left(\prod_{v\in S''\setminus (S\cup \{v_0\})}^{\sim} W_w\otimes \mathbb{Q} \right)^-\times \left(\prod_{v\in S}^{\sim} \Delta G_w \otimes \mathbb{Q}\right)^- \arrow{r}{f_{S,S'}}  & \bigoplus_{v\in S''\setminus\{v_0\}}\mathbb{Q}[G]^-
\end{tikzcd}
\]
Here the left vertical arrow is identity map on $v\in S''\setminus (S\cup \{v_0\})$ and induced by $j_w$ on $v\in S_f$. The right vertical map is the identity map on $v\in S''\setminus (S\cup  \{v_0\})$ and the multiplication by $(1-e_{I_w}\phi_w^{-1})h_v^{-1}$ on $v\in S_f$. This completes the proof.
\end{proof}
Now let us see how our new module is related to the $(S,T)-$ray class group, $Cl_S^T(H)$.
\begin{proposition}
We have the following short exact sequence of $\mathbb{Z}[G]^--$modules.
    \begin{equation}\label{nabla-sequences}
0\xrightarrow[]{} Cl_S^T(H)^-\xrightarrow[]{}\Omega_{S,S'}^T(H)\xrightarrow[]{} Y_{S}^- \oplus Z_{S'}^-\xrightarrow[]{} 0
\end{equation}
Here $Y_S$ is the free $\mathbb{Z}-$module of divisors supported at primes at $S_H$ and for fixed choices of Frobenius $\phi_w\in G$,
 $$Z_{S'}=\bigoplus_{v\in S'} \mathbb{Z}[G/I_w]/(1-\phi_w^{-1}+|I_w|).$$  
\end{proposition}
\begin{proof}
By Proposition \ref{local map - W} and the definition of $f_{S,S'}$, $$coker(f_{S,S'})=\left(\prod_{v\in S}^{\sim} \mathbb{Z}[G_w]/\Delta G_w\right)^- \times \left(\prod_{v\in S'}^{\sim} \mathbb{Z}[G_w/I_w]/(1-\phi_w+|I_w|)\right)^-=Y_S^-\oplus Z_{S'}^-$$
Moreover, observe that the map $f$ is injective. Hence, by the short exact sequence \eqref{sequence-units-class-S}, the definition of $\Omega_{S,S'}^T(H)$ and the previous observation gives the desired result.
\end{proof}
\begin{remark}
    By an argument identical to the proof of Lemma A.3 of \cite{Dasgupta-Kakde} one can show that the definition of $\Omega_{S,S'}^T(H)$ does not depend on the auxiliary set $S''$. Moreover, one can also give an explicit description (in the sense of chapter A.5 of \cite{Dasgupta-Kakde} and chapter 2.7 of \cite{Atsuta-Kataoka-ETNC}) of the extension class corresponding to the short exact sequence \ref{nabla-sequences}. But this will not be needed for the applications in the next two chapters.
\end{remark}

\section{On the equivariant Tamagawa number conjecture}

In this section, we use the minus part of the $p-$adic equivariant tamagawa number conjecture, which we henceforth call $ETNC_p^-$, to calculate the $p-$adic Fitting ideal of our module and establish some other consequences. Here we are using the formulation of $ETNC_p^-$ given in \cite{Atsuta-Kataoka-ETNC}. $ETNC_p^-$ was proved by Bullach-Burns-Daoud-Seo \cite{Bullach-Burns-Daoud-Seo} and Dasgupta-Kakde-Silliman\cite{Dasgupta-Kakde-Silliman-ETNC} using Euler systems techniques and by Gambheera-Popescu \cite{gampheera-popescu-RW} using an equivariant main conjecture for the Ritter-Weiss module and the method of Taylor-Wiles primes.\\

Throughout this section we fix an odd prime $p$. We also keep the setup of section 3. That is, we start with the data $(H/F,S,S',T)$ where $S_{\infty}\subseteq S$. But, we also impose some assumptions on the disjoint sets $S,S'$ and $T$. For that we need the following definitions. If $\mu(H)$ is the group of roots of unity in $H$, define the subgroup
$$\mu(H)^T=\{\zeta\in \mu(H) | \zeta\equiv 1\pmod{v} \text{ for each }v\in T_H\}.$$
Let $\mu(H)_p^T$ be the $p-$Sylow subgroup of $\mu(H)^T$. Moreover, let $S_p$ be the set of primes above $p$ in $F$ and let $S_{wild}^p$ be the set of places $v$ of $F$ such that $p|I_v$ where $I_v$ is the inertia group of $v$ in $H/F$.  

\begin{assumption}\label{assumption}
    $S_{wild}^p \subseteq S\cup S' \cup T$, $S_{wild}^p\cap S_p \subseteq S\cup S'$ and $\mu(H)_p^T$ is trivial.
\end{assumption}

\begin{definition}
    The $\Bbb C[G]$--valued ($G$--equivariant) $L$--function associated to $(H/F,S, T)$ of complex variable $s$ is given by the meromorphic continuation to the entire complex plane of the following holomorphic function
$$\Theta_{S, H/F}^T(s):=\prod_{v\not\in S}(1-e_{I_v}\phi_v^{-1}\cdot Nv^{-s})^{-1}\cdot\prod_{v\in T}(1-e_{I_v}\phi_v^{-1}\cdot Nv^{1-s}), \qquad \text{Re}(s)>0.$$
\end{definition}
Here $\phi_v$ is a choice of Frobenius at the prime $v$. The functional equation of $\Theta_{S,H/F}^T(s)$ implies that in fact it takes values in $\mathbb{C}[G]^-$. The resulting meromorphic continuation (also denoted by $\Theta_{S, H/F}^T(s)$) is holomorphic on $\Bbb C\setminus\{1\}$. We are interested in its special value at $0$. Classical results proved independently by Klingen--Siegel (\cite{Klingen} and \cite{Siegel}) and Shintani \cite{Shintani} show that 
\begin{equation}\label{Klingen-Siegel}\Theta_S^T(H/F):=\Theta_{S, H/F}^T(0)\in\Bbb Q[G]^-.\end{equation}
Now we define a special value of equivariant $L-$function attached to the data $(H/F,S,S',T)$.
\begin{definition}
    For the data $(H/F, S,S',T)$ define the Stickelberger element,
    $$\Theta_{S,S'}^T(H/F):=\Theta_S^T(H/F)\cdot \prod_{v\in S'}(1-e_{I_v}(\phi_v^{-1}-|I_v|))\in \mathbb{Q}[G]^-$$
   
\end{definition}
Observe that when $S=S_{\infty}$ the above value recovers the Stickelberger element $\theta_{S'}^T$ defined in \cite{Atsuta-Kataoka-ETNC}, equation (1.4).\\

In the following theorem, we compute the principal Fitting ideal of $\Omega_{S,S'}^T(H)_p$, in terms of the equivariant $L-$function $\Theta_{S,S'}^T$. Observe that, under Assumption \ref{assumption}, implicitly we are also proving that $\Theta_{S,S'}^T(H/F)\in \mathbb{Z}_p[G]^-$.

\begin{theorem}\label{Fitt-Omega}
    Let $p$ be an odd prime and suppose $S,S'$ and $T$ satisfy Assumption \ref{assumption}. Then, we have the following.

    $$Fitt_{\mathbb{Z}_p[G]^-}(\Omega_{S,S'}^T(H)_p)=(\Theta_{S,S'}^T)$$
\end{theorem}
\begin{proof}
In \cite{Atsuta-Kataoka-ETNC}, Kataoka and Atsuda proved that $ETNC_p^-$ is equivalent to the following.  
$$Fitt_{\mathbb{Z}_p[G]^-}(\Omega_{S_{\infty},S_f\cup S'}^T(H)_p)=(\Theta_{S_{\infty},S_f\cup S'}^T)$$
See Theorem 4.13 in \cite{Atsuta-Kataoka-ETNC} and Remark \ref{Omega-equiv}. Now, by Appendix A in \cite{Dasgupta-Kakde}, $(V^{\theta,-})_p$ is a free $\mathbb{Z}_p[G]^--$module. Therefore, we have the following quadratic presentations.

$$ (V^{\theta,-})_p\xrightarrow[]{\psi_{S,S'}} \bigoplus_{S''\setminus\{v_0\}}\mathbb{Z}_p[G]^-\xrightarrow[]{} \Omega_{S,S'}^T(H)_p \xrightarrow[]{} 0$$

$$ (V^{\theta,-})_p\xrightarrow[]{\psi_{S_{\infty},S_f\cup S'}} \bigoplus_{S''\setminus\{v_0\}}\mathbb{Z}_p[G]^-\xrightarrow[]{} \Omega_{S_{\infty},S_f\cup S'}^T(H)_p \xrightarrow[]{} 0$$

Now, fix a basis for the free $\mathbb{Z}_p[G]^-$-module $(V^{\theta,-})_p$ and use the standard basis for the module in the middle. Let us denote the corresponding matrix for the map $\psi_{S,S'}$ also by $\psi_{S,S'}$. Then, we have $(det(\psi_{S_{\infty},S_f\cup S'}))=(\Theta_{S_{\infty},S_f\cup S'}^T)$ as ideals in $\mathbb{Z}_p[G]^-$. Since for each $v\in S\cup S', $ $h_v$ is a nonzero divisor, so is $\Theta_{S_{\infty},S_f\cup S'}^T$. Hence, there exist $u\in (\mathbb{Z}_p[G]^-)^{\times}$ such that $det(\psi_{S_{\infty},S_f\cup S'})=u\cdot \Theta_{S_{\infty},S_f\cup S'}^T$. Now we can change the basis of $(V^{\theta,-})_p$ such that $u=1$. Now, by Lemma \ref{change of maps} we have the following.

\begin{align*}
    det(\psi_{S,S'})&= \prod_{v\in S_f}\big[(1-e_{I_v}\phi_v^{-1})h_v^{-1}\big] \cdot
    det(\psi_{S_{\infty},S_f\cup S'}) \\
    &= \prod_{v\in S_f} \big[(1-e_{I_v}\phi_v^{-1})h_v^{-1}\big]\cdot \Theta_{S_{\infty},S_f\cup S'}^T  \\
    &= \Theta_{S , S'}^T 
\end{align*}
Therefore, $Fitt_{\mathbb{Z}_p[G]^-}(\Omega_{S,S'}^T(H)_p)=(det(\psi_{S,S'}))=(\Theta_{S,S'}^T)$ as desired.
\end{proof}
The following corollary is a strengthening of Lemma 2.5 of \cite{gambheera-popescu} which shows the $p-$adic integrality of the special values of $G-$equivariant $L-$function associated to $(H/F,S,T)$ under mild conditions.
\begin{corollary}\label{p-integrality}
    Suppose that $S_{wild}^p \subseteq S\cup T$, $S_{wild}^p\cap S_p \subseteq S$ and $\mu(H)_p^T$ is trivial. Then, $\Theta_S^T(H/F)\in \mathbb{Z}_p[G]^-$.
\end{corollary}
\begin{proof}
    Observe that under the given conditions $S, \emptyset$ and $T$ satisfies Assumption \ref{assumption}. Hence, by the previous theorem,
    $$\Theta_S^T(H/F)=\Theta_{S,\emptyset}^T(H/F)\in \mathbb{Z}_p[G]^-$$ as desired.
\end{proof}
By Remark \ref{Omega-equiv}, when $S'=\emptyset$ and $S_{ram}(H/F)\subseteq S\cup T$ we have the following generalization of Theorem 1.14(1) of \cite{gambheera-popescu} which is the $p-$part of the Burns-Kurihara-Sano conjecture. In loc.cit. the statement is written in terms of the Selmer module $Sel_S^T(H)_p^-$. But this Selmer module is a transpose of the Ritter-Weiss module in the sense of section 6.1 of \cite{Dasgupta-Kakde}. Therefore, by Lemma 6.1 of loc.cit. they have the same Fitting ideals. Observe that Assumption \ref{assumption} is much relaxed than the assumptions for the above mentioned theorem.
\begin{corollary} \label{burns-kurihara-sano}
    Let $p$ be an odd prime. Suppose $S_{ram} \subseteq S \cup T$, $S_{wild}^p\cap S_p \subseteq S$ and $\mu(H)_p^T$ is trivial. Then, we have the following.

    $$Fitt_{\mathbb{Z}_p[G]^-}(Sel_S^T(H)_p^-)=Fitt_{\mathbb{Z}_p[G]^-}(\nabla_S^T(H)_p^-)=(\Theta_S^T)$$
\end{corollary}
This corollary for each odd prime $p$ immediately implies a stronger version of the global Burns-Kurihara-Sano conjecture. (See Theorem 1.14(2) of \cite{gambheera-popescu})

\begin{corollary} \label{global-Burns-Kurihara-Sano}
Suppose that $S_{ram} \subseteq S \cup T$, $\mu(H)^T$ is trivial and $T$ does not contain any wildly ramified primes away from 2. Then, we have the following.

    $$Fitt_{\mathbb{Z}[G]^-}(Sel_S^T(H)^-)=Fitt_{\mathbb{Z}[G]^-}(\nabla_S^T(H)^-)=(\Theta_S^T)$$
\end{corollary}

Now let us record a generalization of Kurihara's conjecture. First we show the $p-$part of it. Notice that in the original formulation, (Conjecture 3.2 of \cite{Kurihara}) proved by Dasgupta-Kakde in Appendix B.2 of \cite{Dasgupta-Kakde}, all primes in $T$ were unramified. Here we relax that condition.
\begin{theorem}\label{strong-kurihara}
    Suppose $\mu(H)_p^T$ is trivial and $T\cap S_{wild}^p\cap S_p=\emptyset$ (that is $T$ does not contain any wildly ramified primes above $p$). If $A^T(H)$ is the $p-$Sylow subgroup $Cl^T(H)$, then we have the following equality of ideals in $\mathbb{Z}_p[G]^-$. 
    $$Fitt_{\mathbb{Z}_p[G]^-}(A^T(H)^{\vee,-})= (\Theta_{S_{\infty}}^T)\prod_{v\in S_{ram}(H/F)\setminus T} (N(I_v), 1-e_{I_v}\sigma_v^{-1})$$
\end{theorem}
\begin{proof}{(sketch)}
    The proof of presented in the appendix B.2 of \cite{Dasgupta-Kakde} works word to word with the precise adjustments which we are going to specify now. The proof of Dasgupta-Kakde relies on the Theorem 3.7 of \cite{Dasgupta-Kakde}. One can see that this theorem is still true if we use our conditions of $T$ and if we take $\Sigma$ to be the set $S_{\infty}\cup\{v\in S_{ram};p\mid v\}\setminus T$ which can be proved using Corollary \ref{burns-kurihara-sano} (instead of Theorem 3.3 of \cite{Dasgupta-Kakde} as the authors have done in \cite{Dasgupta-Kakde}). Therefore, Lemma B.7 of \cite{Dasgupta-Kakde} is true with this new $\Sigma$. Now, one can check that the rest of appendix B.2 of \cite{Dasgupta-Kakde} works word to word with our conditions on $T$ and $S_{ram}$ replaced by $S_{ram}\setminus T$. This completes the proof.
\end{proof}
Now, by using the above theorem for each odd prime $p$, we have the following generalization of the Kurihara's conjecture where we allow $T$ to have tamely ramified primes.
\begin{theorem}\label{global-strong-Kurihara}
    Suppose $\mu(H)^T$ is trivial and $T$ does not contain any wildly ramified primes away from 2. Then we have the following equality of ideals in $\mathbb{Z}[G]^-$. 
    $$Fitt_{\mathbb{Z}[G]^-}(Cl^{T}(H)^{\vee,-})= (\Theta_{S_{\infty}}^T)\prod_{v\in S_{ram}(H/F)\setminus T} (N(I_v), 1-e_{I_v}\sigma_v^{-1})$$
\end{theorem}
We end the section with the following technical proposition which will be used in the calculations of the next section.
\begin{proposition}\label{equality}
 Let $p$ be an odd prime and suppose $S, S'$ and $T$ satisfy Assumption \ref{assumption}. If $A_S^T$ is the $p-$Sylow subgroup $Cl_S^T(H)$ then we have the following equality of ideals in $\mathbb{Z}_p[G]^-$.

$$Fitt_{\mathbb{Z}_p[G]^-}(A_S^{T,-})\prod_{v\in S_f}(\sigma_v-1)=(\Theta_{S,S'}^T)Fitt_{\mathbb{Z}_p[G]^-}^{[1]}((Z_{S'})_p^-)Fitt_{\mathbb{Z}_p[G]^-}(M_p^-)$$
Here, $\sigma_v$ is some choice of Frobenius of the place $v$ and $M=\prod_{v\in S_f} M_v$ where $M_v$ is the following quotient of the ideals of $\mathbb{Z}[G]^-$. $$M_v=(\sigma_v-1, i-1; i\in I_v)/(\sigma_v-1)$$ 
\end{proposition}
\begin{proof}
    By tensoring the short exact sequence \ref{nabla-sequences} with $\mathbb{Z}_p$ and observing that $Y_{S_{\infty}}^-=0$, we get the following.
    $$0\xrightarrow[]{} A_S^T(H)^-\xrightarrow[]{}\Omega_{S,S'}^T(H)_p\xrightarrow[]{} (Y_{S_f})_p^- \oplus (Z_{S'})_p^-\xrightarrow[]{} 0$$
    Let $N=\prod_{v\in S'} N_v$ where for each $v\in S'$, 
    $$N_v=(1-\sigma_v^{-1}+|I_v|,i-1;i\in I_v)/(1-\sigma_v^{-1}+|I_v|).$$ We also have the following short exact sequence.
    $$0\xrightarrow[]{} M_p^-\oplus N_p^-
 \xrightarrow[]{}\prod_{v\in S_f}\mathbb{Z}_p[G]^-/(1-\sigma_v)\times\prod_{v\in S'}\mathbb{Z}_p[G]^-/(1-\sigma_v^{-1}+|I_v|) \xrightarrow[]{} (Y_{S_f})_p^- \oplus (Z_{S'})_p^-\xrightarrow[]{} 0$$
    Observe that in the two short exact sequences above, the modules in the middle are quadratically presented. Then, by Lemma 2.7 in \cite{Dasgupta-Kakde}, we have,

\begin{equation}\label{long-equation}
    \begin{aligned}
        Fitt_{\mathbb{Z}_p[G]^-}(A_S^{T,-})\prod_{v\in S_f}(\sigma_v-1)\prod_{v\in S'}(1-\sigma_v^{-1}+|I_v|)=Fitt_{\mathbb{Z}_p[G]^-}& (\Omega_{S,S'}^T(H)_p)\cdot Fitt_{\mathbb{Z}_p[G]^-}(N_p^-) \\
& \times      Fitt_{\mathbb{Z}_p[G]^-}(M_p^-)
    \end{aligned}
\end{equation}
On the other hand, we have the following short exact sequence.
$$0\xrightarrow[]{}  N_p^-
 \xrightarrow[]{} \prod_{v\in S'}\mathbb{Z}_p[G]^-/(1-\sigma_v^{-1}+|I_v|) \xrightarrow[]{} (Z_{S'})_p^-\xrightarrow[]{} 0$$
Observe that,

$$Fitt_{\mathbb{Z}_p[G]^-}\Big(\prod_{v\in S'}\mathbb{Z}_p[G]^-/(1-\sigma_v^{-1}+|I_v|)\Big)=\prod_{v\in S'}(1-\sigma_v^{-1}+|I_v|)$$
Now, since $1-\sigma_v^{-1}+|I_v|$ is clearly a nonzero divisor in $\mathbb{Z}_p[G]^-$ for all $v$, the module, by Proposition 4.9 of \cite{gambheera-popescu} $\prod_{v\in S'}\mathbb{Z}_p[G]^-/(1-\sigma_v^{-1}+|I_v|)$ has projective dimension $1$. Hence, we have,
$$Fitt_{\mathbb{Z}_p[G]^-}(N_p^-)=\prod_{v\in S'}(1-\sigma_v^{-1}+|I_v|)\cdot Fitt_{\mathbb{Z}_p[G]^-}^{[1]}(Z_{S'})_p^-$$
By substituting this in equation \ref{long-equation}, applying Theorem \ref{Fitt-Omega} and by canceling out the term $\prod_{v\in S'}(1-\sigma_v^{-1}+|I_v|)$ we get the desired equation.
\end{proof}

\section{Iwasawa theoritic considerations}

In this section we use the results proved in section 4 together with some technical tools developed in section 2 to compute the Fitting ideals of some interesting equivariant Iwasawa modules living at the top of the cyclotomic Iwasawa tower of a CM field. We also give a conjectural description to the Fitting ideal of $X^-$ as mentioned in the introduction.\\

Let us start with our Iwasawa theoretic set up. Let $p$ be an odd prime and let $H/F$ be a finite, abelian CM extension of a totally real number field $F$ with Galois group $G:=Gal(H/F)$. Let $H_{\infty}$ be the cyclotomic $\mathbb{Z}_p-$extension of $H$. Denote $\Gamma:=Gal(H_{\infty}/H)$ and $\mathcal{G}:=Gal(H_{\infty}/F)$. Let $S$ and $T$ be two nonempty, disjoint sets of places in $F$ on which we impose the following assumption.

\begin{assumption}\label{assumption-Iwasawa}
    $S_{\infty}\subseteq S$ and $T\cap S_p=\emptyset$.
\end{assumption}
 
As usual, $S_{ram}(H_\infty/F)$ denotes the ramification locus of $H_\infty/F$. Note that $S_p\subseteq S_{ram}(H_\infty/F)$. Further, we let $H_n$ denote the $n$--th layer of the cyclotomic extension $H_\infty/H$, and let $G_n:=Gal(H_n/F)$. As usual, we define the topological, profinite group algebras (the classical Iwasawa algebra and the $G$--equivariant Iwasawa algebra respectively)
$$\Lambda:=\Bbb Z_p[[\Gamma]]:=\varprojlim\limits_n\Bbb Z_p[Gal(H_n/H)]$$
$$\Bbb Z_p[[\mathcal G]]:=\varprojlim\limits_n\Bbb Z_p[G_n]$$
where the transition maps $\Tilde{\pi}_n^{n+1}:{\Bbb Z}_p[Gal(H_{n+1}/H)]\to\Bbb Z_p[Gal(H_n/H)]$ and $\pi_n^{n+1}:{\Bbb Z}_p[G_{n+1}]\to\Bbb Z_p[G_n]$ in the projective limit are  induced by Galois restriction.
\\

Observe that since $H$ is a CM field, $H_n$ and $H_\infty$ are also CM fields, and therefore there is a unique complex conjugation automorphism $c\in \mathcal G$, which restricts to the unique complex conjugation in $G_n$ (which is also denoted by $c$), for all $n\geq 0$. Throughout this chapter, for any $\Bbb Z_p[[\mathcal G]]-$module $M$, we define,
$$M^-:=\frac{1}{2}(1-c)\cdot M.$$
Observe that this is a $\Bbb Z_p[[\mathcal G]]^-$--module, 
where 
$$\Bbb Z_p[[\mathcal G]]^-:=\frac{1}{2}(1+c)\cdot\Bbb Z_p[[\mathcal G]]\cong \Bbb Z_p[[\mathcal G]]/(1+c).$$ 
Now, due to a well known restriction--inflation property of Artin $L$--functions, we have 
$$\pi_n^{n+1}(\Theta_S^T(H_{n+1}/F)=\Theta_S^T(H_n/F), \qquad \text{ for all } n\geq 0.$$
In addition to Assumption \ref{assumption-Iwasawa} if we further assume that $S_{ram}(H_{\infty}/F)\subseteq S\cup T$, by Corollary \ref{p-integrality} we have that, for each $n$, $\Theta_S^T(H_n/F)\in \Bbb Z_p[G_n]^-$. These properties allows us to define the $(S,T)-$modified equivariant $p-$adic L-function:
$$\Theta_S^T(H_{\infty}/F):=(\Theta_S^T(H_n/F))_n\,\in\, \varprojlim\limits_n\Bbb Z_p[G_n]^-=\mathbb{Z}_p[[\mathcal{G}]]^-$$
When there is no risk of confusion we simply write $\Theta_S^T(H_{n})$ and $\Theta_S^T(H_{\infty})$ in place of $\Theta_S^T(H_{n}/F)$ and $\Theta_S^T(H_{\infty}/F)$ respectively.

\begin{definition}
    Define the $\Bbb Z_p[[\mathcal G]]-$modules $X_S^T:=\varprojlim_n A_S^T(H_n)$, $X^T:=\varprojlim_n A^T(H_n)$ and  $X:=\varprojlim_n A(H_n)$ where $A(H_n)$ is the $p$-Sylow subgroup of $Cl(H_n)$. Here the projective limits are taken under the norm map.
\end{definition}
The following proposition says that $X_S^T$ is a genuine equivariant Iwasawa module. Therefore, its Fitting ideal over $\Bbb Z_p[[\mathcal G]]$ makes sense.
\begin{proposition}\label{fgTorsion}
    $X_S^T$ is finitely generated and torsion as a $\Lambda-$module. Hence, it is also a finitely generated and torsion as a $\Bbb Z_p[[\mathcal G]]-$module.
\end{proposition}
\begin{proof}
  In the proof of Proposition 4.5 of \cite{gampheera-popescu-RW}, we prove this result.
\end{proof}

Now let us recall Wiles work on the Iwasawa main conjecture for totally real fields. Wiles' main theorem (Theorem 1.2 of \cite{Wiles}) can be restated as follows.

\begin{theorem}\label{Wiles}
    Let $\mathbb{Q}_p(\mathbb{Z}_p[[\mathcal{G}]]^-)=\mathbb{Z}_p[[\mathcal{G}]]^-\otimes_{\mathbb{Z}_p}\mathbb{Q}_p$. Suppose that $\zeta_p \in H$ where $\zeta_p$ is a primitive $p-$th root of unity and $F_{\infty}\cap H=F$. Moreover assume that $T\cap S_{ram}(H_{\infty}/F)=\emptyset$. Then, we have the following equality of ideals of $\mathbb{Q}_p(\mathbb{Z}_p[[\mathcal{G}]]^-)$.
    $$Fitt_{\mathbb{Q}_p(\mathbb{Z}_p[[\mathcal{G}]]^-)}(X^{T,-}\otimes_{\mathbb{Z}_p}\mathbb{Q}_p)=(\Theta_{S_p\cup S_{\infty}}^T(H_{\infty}))$$
\end{theorem}
\begin{proof}
     By tensoring the exact sequence (10) of \cite{Greither-Popescu} at the level of $H_n$ with $\mathbb{Z}_p$ and by taking the minus part we have the following.

    $$0\xrightarrow[]{} \mu(H_n)_p \xrightarrow[]{} \big(\bigoplus_{w\in T_{H_n}}\kappa(w)^{\times}\otimes \mathbb{Z}_p \big)^- \xrightarrow[]{} A^{T,-}(H_n)\xrightarrow[]{} A^-(H_n)\xrightarrow[]{} 0$$
    Here $\kappa(w)$ is the residue field associated to the prime $w$. Now, observe that we have the isomorphisms of $\mathbb{Z}[G_n]-$modules, 
    $$\bigoplus_{w\in T_{H_n}}\kappa(w)^{\times} \cong \bigoplus_{v\in T}\kappa(w(v))^{\times} \otimes _{\mathbb{Z}[G_{n,v}]} \mathbb{Z}[G_{n}]$$
    Here $\kappa(w(v))$ is a prime in $H_n$ that is above $v$ and $G_{n,v}$ is the decomposition group of $v$ in $H_n/F$. It is easy to see that as a $\mathbb{Z}[G_{n,v}]-$module, we have
    $$\kappa(w(v))^{\times}\cong \mathbb{Z}[G_{n,v}]/(1-\sigma_{v,n}^{-1}\cdot\textbf{N}v)$$
    where $\textbf{N}v$ is the norm of the prime $v$ and $\sigma_{v,n}$ is the Frobenius of $v$ in $G_{n,v}$. Therefore, we have,
    $$\big(\bigoplus_{w\in T_{H_n}}\kappa(w)^{\times}\otimes \mathbb{Z}_p \big)^- \cong \bigoplus_{v\in T}\mathbb{Z}_p[G_{n}]^-/(1-\sigma_{v,n}^{-1}\cdot\textbf{N}v)$$
    Now, by taking the projective limit of the above exact sequence on the Iwasawa tower with respect to the norm maps between levels, and then by tensoring by $\mathbb{Q}_p$, we have the exact sequence of $\mathbb{Q}_p(\mathbb{Z}_p[[G]]^-)-$modules,
    $$0\xrightarrow[]{} \mathbb{Q}_p(1)\xrightarrow[]{} \bigoplus_{v\in T}\mathbb{Q}_p(\mathbb{Z}_p[[\mathcal{G}]]^-)/(1-\sigma_{v,\infty}^{-1}\cdot\textbf{N}v) \xrightarrow[]{} X^{T,-}\otimes_{\mathbb{Z}_p}\mathbb{Q}_p\xrightarrow[]{} X^- \otimes_{\mathbb{Z}_p}\mathbb{Q}_p\xrightarrow[]{} 0$$
    where $\sigma_{v,\infty}\in \mathcal{G}$ is the Frobenius of $v$ in $H_{\infty}/F$. Here $\mathbb{Q}_p(1)$ is the additive group $\mathbb{Q}_p$ where $\mathcal{G}$ acts via the $p-$cyclotomic character $c_p:\mathcal{G}\xrightarrow[]{} Aut(\mu(H_{\infty})_p)=\mathbb{Z}_p^{\times}$. Now, clearly we have,
    $$Fitt_{\mathbb{Q}_p(\mathbb{Z}_p[[\mathcal{G}]]^-) }\big(\bigoplus_{v\in T}\mathbb{Q}_p(\mathbb{Z}_p[[\mathcal{G}]]^-)/(1-\sigma_{v,\infty}^{-1}\cdot\textbf{N}v)\big)=\prod_{v\in T}(1-\sigma_{v,\infty}^{-1}\cdot\textbf{N}v)$$
    Now, by Lemma 5.21(2), Lemma 5.13(2) of \cite{Greither-Popescu} and Theorem 1.2 of \cite{Wiles} we have,
    $$(\Theta_{S_p\cup S_{\infty}}^{\emptyset}(H_{\infty}))Fitt_{\mathbb{Q}_p(\mathbb{Z}_p[[\mathcal{G}]]^-) }\mathbb{Q}(1)=Fitt_{\mathbb{Q}_p(\mathbb{Z}_p[[\mathcal{G}]]^-) }(X^-\otimes_{\mathbb{Z}_p}\mathbb{Q}_p).$$
    Since Fitting ideals over $\mathbb{Q}_p(\mathbb{Z}_p[[G]]^-)$ are multiplicative, this together with the fact that,
    $$\Theta_{S_p\cup S_{\infty}}^{\emptyset}(H_{\infty}) \prod_{v\in T}(1-\sigma_{v,\infty}^{-1}\cdot\textbf{N}v)= \Theta_{S_p\cup S_{\infty}}^{T}(H_{\infty})$$
    completes the proof.
\end{proof}
\begin{remark}\label{not-equivariant}
    Notice that when $F_{\infty}\cap H=F$, we have the ring isomorphism $$\mathbb{Q}_p(\mathbb{Z}_p[[\mathcal{G}]])\cong \bigoplus_{[\chi]\in \Hat{G}} \mathbb{Q}_p(\chi)[[\Gamma]].$$ 
    The summation runs through equivalence classes of $\overline{\mathbb{Q}}_p^{\times}$  characters of $G$. Here, $\mathbb{Q}_p(\chi)$ is the extension obtained by adjoining $\mathbb{Q}_p$ with all the values of the character $\chi$. Therefore, despite the appearances, the previous theorem, which is essentially a reformulation of Theorem 1.2 of \cite{Wiles}, is neither integral (since the module is tensored by $\mathbb{Q}_p$) nor equivariant (it is rather a statement for an odd Artin charater of $F$) in nature. Therefore, as the main goal of this paper we are proving an integral, equivariant refinement of the theorem above. Namely, we calculate $Fitt_{\mathbb{Z}_p[[\mathcal{G}]^-}(X^{T,-})$ in Corollary \ref{main corollary}. In fact we prove a stronger result by calculating $Fitt_{\mathbb{Z}_p[[\mathcal{G}]^-}(X_S^{T,-})$ in Theorem \ref{main theorem}.
\end{remark}

\begin{proposition} \label{LHS}
    For all $v\in S_f$, let $\sigma_{v,\infty}\in \mathcal{G_v}$ be a choice of Frobenius with infinite order. Define, $\sigma_{v,n}:=\pi_n(\sigma_{v,\infty})\in G_{v,n}$ which is also a choice of Frobenius in $G_n$. Here, $\pi_n$ is the Galois restriction. Under the Assumption \ref{assumption-Iwasawa} we have the following equality of ideals in $\mathbb{Z}_p[[\mathcal{G}]]^-$.
    $$\varprojlim_n (Fitt_{\mathbb{Z}[G_n]^-}(A_S^{T,-}(H_n)) \prod_{v\in S_f} (\sigma_{v,n}-1))=Fitt_{\mathbb{Z}_p[[\mathcal{G}]]^-} (X_S^{T,-})\prod_{v\in S_f} (\sigma_{v,\infty}-1)$$
\end{proposition}
\begin{proof}
    Notice that, for any finite prime $v$, the decomposition group $\mathcal{G}_v$ in $H_{\infty}/F$ has $\mathbb{Z}_p$ rank 1. Hence, we can pick a Frobenius $\sigma_{v,\infty}$ with infinite order. First, let us prove that the left hand side is well defined. That is, we have well defined transition maps between the ideals for large enough $n$.\\\\
    We know that when $n\gg 0$, $H_{n+1}/H_n$ is totally ramified at each prime over $p$ and all the $T-$primes are inert. Observe that we have the following commutative diagram.
    
\[
\begin{tikzcd}
    \bigoplus_{w\in T_{H_{n+1}}} \kappa(w)^{\times} \arrow{d}\arrow{r} & Cl^T(H_{n+1}) \arrow{d}\arrow{r} & Cl(H_{n+1}) \arrow{d}\arrow{r} & 0 \\%
    \bigoplus_{v\in T_{H_{n}}} \kappa(v)^{\times} \arrow{r} & Cl^T(H_{n}) \arrow{r} & Cl(H_{n}) \arrow{r} & 0
\end{tikzcd}
\]   
Here $\kappa(w)$ and $\kappa(v)$ are residue fields associated to the primes $w$ and $v$ respectively. The rows are obtained from exact sequence (10) in \cite{Greither-Popescu}. The vertical arrows are induced by the norm map. Now, since, the $T-$primes are inert in $H_{n+1}/H_n$, we have $Gal(\kappa(w)/\kappa(v))=Gal(H_{n+1}/H_n)$. Therefore, the left vertical map is surjective. Since the primes above $p$ are totally ramified in $H_{n+1}/H_n$, a standard argument in class field theory implies that the right vertical map is also surjecive. Therefore, so is the middle vertical map. \\\\
Now, observe that we also have the following natural commutative diagram whose vertical arrows are induced by norm maps and the horizontal maps are surjective.

\[
\begin{tikzcd}
    Cl^T(H_{n+1}) \arrow{d}\arrow{r} & Cl_S^T(H_{n+1}) \arrow{d} \\%
     Cl^T(H_{n}) \arrow{r} & Cl_S^T(H_{n}) 
\end{tikzcd}
\]
Now, since we proved that the left vertical arrows are surjective, so is the right vertical arrow. Now, by tensoring by $\mathbb{Z}_p$ and taking the minus part, we have that the norm map $A_S^{T,-}(H_{n+1})\xrightarrow[]{} A_S^{T,-}(H_{n})$ is surjective. Therefore, by the properties of Fitting ideals, we have $$\pi_n^{n+1}(Fitt_{\mathbb{Z}_p[G_{n+1}]^-}A_S^{T,-}(H_{n+1}))\subseteq Fitt_{\mathbb{Z}_p[G_{n}]^-}A_S^{T,-}(H_{n}).$$ On the other hand, for each $v\in S_f$ we have $$\pi_n^{n+1}((\sigma_{n+1,v}-1))=(\sigma_{n,v}-1).$$ Therefore, Galois restriction induces well defined transition maps between the modules in the left hand side for $n\gg 0$.\\\\
Hence, by Lemma \ref{Proj-ideal-lemma} we have,
$$\varprojlim_n (Fitt_{\mathbb{Z}[G_n]^-}(A_S^{T,-}(H_n)) \prod_{v\in S_f} (\sigma_{v,n}-1))=\varprojlim_n (Fitt_{\mathbb{Z}[G_n]^-}(A_S^{T,-}(H_n))) \prod_{v\in S_f} (\sigma_{v,\infty}-1).$$
Now, by Proposition \ref{fgTorsion}, we have $X_S^{T,-}$ is finitely generated and torsion as a $\Lambda-$module. This together with the surjectivity of the norm maps and Corollary 2.3 in \cite{gampheera-popescu-RW} implies the following.
$$\varprojlim_n (Fitt_{\mathbb{Z}[G_n]^-}(A_S^{T,-}(H_n))=Fitt_{\mathbb{Z}[[\mathcal{G}]]^-}(X_S^{T,-})$$
This completes the proof.
\end{proof}

For each $n$, let $G_{p,n}$ be the Sylow $p-$subgroup of $G_n$. Then, $\mathcal{G}\cong G'\times G_{p,\infty}$ where $G_{p,\infty}=\varprojlim_n G_{p,n}$ and $G'$ is the non$-p$ part of $G$. Let $G_p$ be the torsion part of $G_{p,\infty}$. Then, $\mathcal{G}\cong G'\times G_p\times \Gamma'$, where $\Gamma'$ is a subgroup of $\mathcal{G}$ that is isomorphic to $\mathbb{Z}_p$. Hence, there is also a subfield $H'$ of $H_{\infty}$, namely the fixed field of $\Gamma'$, whose cyclotomic $\mathbb{Z}_p-$extension is $H_{\infty}$, such that $Gal(H_{\infty}/H')=\Gamma'$. Check the proof of Theorem 4.2(2) in \cite{Greither-Popescu} for details.\\

Now, observe that if $H'_n$ is the $n$th layer of the cyclotomic $\mathbb{Z}_p-$extension of $H'$, we have $X_S^T=\varprojlim_n A_S^T(H'_n)$. In other words, $X_S^T$ does not depend on the base field. So, for the purpose of computing, $Fitt_{\mathbb[[\mathcal{G}]]^-}(X_S^{T,-})$, Without loss of generality we can assume that $H'=H$. In other words, we may assume that $\mathcal{G}\cong G\times \Gamma$.
\\\\
Let $v$ be a prime above $p$ and let $\mathcal{I}_v$ be its inertia group in the extension $H_{\infty}/F$. Now, since $rank_{\mathbb{Z}_p}(\mathcal{I}_{v})=1$, we have $\mathcal{I}_{v}=Tor(\mathcal{I}_{v})\times \overbar{\langle y \rangle}$ where $y=(g,\gamma^t)\in G\times \Gamma$. Here $\overbar{\langle y \rangle}$ is the closure of the subgroup of $\mathcal{G}$ generated by $y$ and $\gamma$ is a topological generator of $\Gamma$. Moreover we are viewing $\Gamma$ as a multiplicative group. Observe that we can choose $\gamma$ and $y$ such that $t$ is a power of $p$ and $g$ has a $p-$power order. Now, we also choose a large $n$ to make sure that $\gamma^{p^n}\in \langle y \rangle$.\\\\
Now, we know that, if $\pi_n$ is the Galois restriction,
$$\pi_n:\mathcal{G}\cong G\times \Gamma\xrightarrow[]{} G_n\cong G\times (\Gamma/\Gamma^{p^n})$$
where under relevant identifications $\pi_n(G)=G$ and $\pi_n(\Gamma)=\Gamma/\Gamma^{p^n}$. Therefore, if $I_{n,v}$ is the inertia group of $v$ in the extension $H_n/F$, we have $$I_{n,v}=\pi_n(\mathcal{I}_{v})\cong \pi_n(Tor(\mathcal{I}_{v}))\times (\overbar{\langle y \rangle}/\Gamma^{p^n}).$$ Observe that, here $\pi_n(Tor(\mathcal{I}_{v}))\subseteq G$ is independent of $n$. Now, choose generators of $\pi_n(Tor(\mathcal{I}_{v}))$ such that,
$$\pi_n(Tor(\mathcal{I}_{v}))=\langle g_1\rangle\times \langle g_2\rangle\times\text{...  }\langle g_{s-1}\rangle$$
and let $g_s^{(n)}=\pi_n(y)$. Then we have $$I_{n,v}=\langle g_1\rangle\times \langle g_2\rangle\times\text{...  }\langle g_s^{(n)}\rangle $$
Hence, under this identification, for $n\gg 0$, $\pi_n^{n+1}(g_i)=g_i$ and $\pi_n^{n+1}(g_s^{(n+1)})=g_s^{(n)}$. Let us define $I_r^{(v)}:= \langle g_r\rangle \subseteq \mathcal{G}$ for each $1 \leq r \leq s-1$.\\
\begin{definition}\label{J-definition-p}
   Let $v$ be a prime of $F$ that is above $p$ and let $\mathcal{G}_v$ be the decomposition group of $v$ in the extension $H_{\infty}/H$. Define the ideals in $\mathbb{Z}_p[[\mathcal{G}]]$, 
    $$\Tilde{Z}_{i,v}^{\infty}:=\left(\prod_{j=1}^{s-i} N(I_{l_j}^{(v)}); 1\leq l_1 < l_2 < \text{ ... } , < l_{s-i}\leq s-1\right)$$
     $$\Delta {\mathcal{G}_v}:=Ker(\mathbb{Z}_p[[\mathcal{G}]]\xrightarrow[]{} \mathbb{Z}_p[[\mathcal{G}/\mathcal{G}_v]])$$
$$\Tilde{J}_v^{\infty}:=\sum_{i=1}^s \Tilde{Z}_{i,v}^{\infty} \Delta {\mathcal{G}_v} $$
\end{definition}
Moreover, if $v$ is a prime of $F$ away from $p$, it does not ramify in the extension $H_{\infty}/H$. Hence, under the identification of $\mathcal{G}$ with $G\times \Gamma$, inertia group of $v$ in the extensions $H_{\infty}/F$ and $H_n/F$ for each $n$, is the same subgroup (say $I_v$) of $G$. Henceforth, for such $v$, we denote its cyclic decomposition as,
$$I_v=I_1^{(v)}\times I_2^{(v)}\times..... I_s^{(v)}.$$
\begin{definition}
If $v$ of $F$ that is away from $p$, we define,
$$Z_{i,v}^{\infty}:=\left(\prod_{j=1}^{s-i} N(I_{l_j}); 1\leq l_1 < l_2 < \text{ ... } , < l_{s-i}\leq s\right)$$
$$J_v^{\infty}:=\sum_{i=1}^s Z_{i,v}^{\infty} \Delta{\mathcal{G}_v}^{i-1}.$$
Here also $\mathcal{G}_v$ is the decomposition group of $v$ in $H_{\infty}/F$ and the augmentation ideal $\Delta \mathcal{G}_v$ is defined similar to Definition \ref{J-definition-p}.  
\end{definition}
\begin{definition}
    Let $v$ be a prime of $F$ that is not above $p$. The fractional ideal $Q_v$ of $\mathbb{Z}_p[[\mathcal{G}]]^-$ is defiend as,
    $$Q_v:=\big(N(I_v),( 1-e_{\mathcal{I}_v}\sigma_{v,\infty}^{-1})J_v^{\infty}\big)$$
\end{definition}
When $v$ is away from $p$, its inertia group in $H_{\infty}/F$ is finite. Therefore, the above definition makes sense.
\begin{definition}
    If $v\in S\cap S_{ram}(H_{\infty}/F)$, then define the (possibly fractional) ideal of $\mathbb{Z}_p[[\mathcal{G}]]^-$ as,
    $$R_v:=\left(N(A)\Delta B^{r_B-2}\,\bigg\mid\, A, B\text{ subgroups of } \mathcal{G}_v,\text{ with }\mathcal G_v=A\times B \text{ and }  A \text{ finite }\right)$$
Here $A, B$ runs through all the possibilities such that $\mathcal{G}_v=A\times B$ where $A$ is finite. $\Delta B=(g-1\text{ ; }g\in B)$ is the augmentation ideal of $B$ and $r_B$ is the minimum number of generators of $B$.
\end{definition}

\begin{proposition}\label{RHS}
Suppose $S,T$ satisfies Assumption \ref{assumption-Iwasawa} and $S'$ be a set of primes in $F$ disjoint from $S\cup T$ which satisfies $S_{ram}(H_{\infty}/F)\subseteq T\cup S\cup S'$. For each prime $v$ of $F$ fix a Frobenius of $v$, $\sigma_{v,\infty}\in \mathcal{G}$ in the extension $H_{\infty}/F$ and let $\sigma_{v,n}=\pi_n(\sigma_{v,\infty})$. For each $n$ and the set $S_f$, denote by $M(H_n)$ the module $M$ defined in Proposition \ref{equality} with respect to the choices of Frobenius $\sigma_{v,n}$ for each $v\in S_f$. Let $M(H_{\infty})_p=\varprojlim_n M(H_n)_p$ where the transition maps are induced by the Galois restriction. Then, we have the following equality of ideals in $\mathbb{Z}_p[[\mathcal{G}]]^-$.

    \begin{align*}
       \varprojlim_n ((\Theta_{S,S'}^T(H_n))Fitt_{\mathbb{Z}_p[G_n]^-}^{[1]}((Z_{S'}(H_n))_p^-)Fitt_{\mathbb{Z}_p[G_n]^-}(M(H_n)_p^-))\\
       =\Theta_{S\cup S'_{p}}^T(H_{\infty})\prod_{v\in S'\setminus S_p} Q_v \prod_{v\in S'_{p}} \Tilde{J}_v^{\infty}Fitt_{\mathbb{Z}_p[[\mathcal{G}]]^-}(M(H_{\infty})_p^{-}) 
    \end{align*}
Here $S_p'=S'\cap S_{p}$.
\end{proposition}
\begin{proof}
    First let us prove that the left hand side is well defined. For a prime $v\in S'$, let $\sigma_{v,\infty}\in \mathcal{G}$ be a choice of Frobenius and let $\sigma_{v,n}:=\pi_n(\sigma_{v,\infty})$. Moreover, let $$h_{v,n}:=1-e_{I_{v,n}}\sigma_{v,n}^{-1}+|I_{v,n}|$$ where $I_{v,n}$ is the inertia group of $v$ in the extension $H_n/F$. Observe that, by Theorem \ref{A-shifted-Fitt} for each $n$ we have,
    $$\Big(\prod_{v\in S'} h_{v,n} \Big)Fitt_{\mathbb{Z}_p[G_n]^-}^{[1]}(Z_{S'}(H_n)_p^-) \subseteq \prod_{v\in S'}(N(I_{v,n}), 1-e_{I_{v,n}}\sigma_{v,n}^{-1}).$$
Therefore, using Theorem \ref{strong-kurihara}, we have

\begin{align*}  (\Theta_{S,S'}^T(H_n))Fitt_{\mathbb{Z}_p[G_n]^-}^{[1]}(Z_{S'}(H_n)_p^-)&=(\Theta_{S,\emptyset}^T(H_n))\Big(\prod_{v\in S'} h_{v,n} \Big)Fitt_{\mathbb{Z}_p[G_n]^-}^{[1]}(Z_{S'}(H_n)_p^-)\\
&\subseteq \Theta_{S_{\infty}}^T(H_n) \prod_{v\in S_f} (1-e_{I_{v,n}}\sigma_{v,n}^{-1})\prod_{v\in S'}(N(I_{v,n}), 1-e_{I_{v,n}}\sigma_{v,n}^{-1}) \\
    &\subseteq Fitt_{\mathbb{Z}_p[G_n]^-}(A^T(H_n)^{\vee,-})
\end{align*}
Hence, $J_n:=(\Theta_{S,S'}^T(H_n))Fitt_{\mathbb{Z}_p[G_n]^-}^{[1]}(Z_{S'}(H_n)_p^-)$ is an ideal of $\mathbb{Z}_p[G_n]^-$. Now, let us prove that the transition maps are well defined. That is, $\pi_n^{n+1}(J_{n+1})\subseteq J_n$. For this, we focus on generators. First observe that, $\pi_n^{n+1}(\Theta_{S,\emptyset}^T(H_{n+1}))=\Theta_{S,\emptyset}^T(H_{n})$.\\\\
Now, we know that $Z_{S'}(H_n)_p=\bigoplus_{v\in S'} A_v$ where $$A_v=\mathbb{Z}_p[G/I_{v,n}]/(1-\sigma_{v,n}^{-1}+|I_{v,n}|).$$
Now, let us look at the generators of $J_n$. These can be obtained by looking at the generators of the fractional ideals $Fitt_{\mathbb{Z}_p[G_n]^-}^{[1]}(A_v)$ for each $v\in S'$. Now, for a fixed $v\in S'$, and the given cyclic decomposition of $I_{v,n}$ we consider the terms $e(\lambda,\mu)$ as in Lemma \ref{Shifted-Fitt-generators}. For each $n$ and fixed tuples $\lambda$ and $\mu$ let us denote them as $e_n(\lambda,\mu)$. Observe that for $n\gg 0$,
\begin{itemize}
    \item $\pi_n^{n+1}(e_{n+1}(\lambda,\mu))=e_{n}(\lambda,\mu)$ if $v$ is not above $p$, or $\lambda$  does not contain $s$. Let us call these good generators.
    \item $\pi_n^{n+1}(e_{n+1}(\lambda,\mu))=p\cdot e_{n}(\lambda,\mu)$ if else. Let us call these bad generators.
\end{itemize}
Moreover, for $n\gg 0$ we have $\pi_n^{n+1}(N(I_{v,n+1}))=p\cdot N(I_{v,n})$ or $ N(I_{v,n})$ depending on whether $v$ is above $p$ or not. For each $n$ and prime $v$ of $F$ we also have, $$\pi_n^{n+1}(1-e_{I_{v,n+1}}\sigma_{v,n+1}^{-1})=1-e_{I_{v,n+1}}\sigma_{v,n}^{-1}.$$ Therefore, clearly, $\pi_n^{n+1}(J_{n+1})\subseteq J_n$ for each $n$. On the other hand, clearly we have that the projection map $M(H_{n+1})_p^-\xrightarrow[]{} M(H_{n})_p^-$ induced by $\pi_n^{n+1}$ is surjective. Hence, we have $$\pi_n^{n+1}(Fitt_{\mathbb{Z}_p[G_{n+1}]^-}(M(H_{n+1})_p^-)\subseteq Fitt_{\mathbb{Z}_p[G_{n}]^-}(M(H_{n})_p^-).$$ Therefore, the left hand side is well defined.\\

Moreover, from the previous calculation and Lemma \ref{ignore-generators}, we can ignore the bad generators and $N(I_v)$ terms for primes $v$ above $p$, when computing the projective limits. Using this observation and Theorem \ref{A-shifted-Fitt}, we have

\begin{eqnarray*}
    \varprojlim_n (\Theta_{S,S'}^T(H_n))Fitt_{\mathbb{Z}_p[G_n]^-}^{[1]}(Z_{S'}(H_n)_p^-)&=&\Theta_{S}^T(H_{\infty})\prod_{v\in S'\setminus S_p} Q_v \prod_{v\in S_{p}'} (1-e_{\mathcal{I}_v}\sigma_{v,\infty}^{-1})\Tilde{J}_v^{\infty}  \\
  &=& \Theta_{S\cup S_{p}'}^T(H_{\infty})\prod_{v\in S'\setminus S_p} Q_v \prod_{v\in S_{p}'} \Tilde{J}_v^{\infty}
\end{eqnarray*}

Now, since the remaining generators align under the ransition maps $\pi_n^{n+1}$, by Lemma \ref{Proj-ideal-lemma} and Lemma \ref{ignore-generators}, we can take the projective limit on the left hand side separately for each ideal. Therefore, we are left with proving the following.\\
$$\varprojlim_n Fitt_{\mathbb{Z}_p[G_n]^-}(M(H_n)_p^-)=Fitt_{\mathbb{Z}_p[[\mathcal{G}]]^-}(M(H_{\infty})_p^-)$$
Observe that we have the following short exact sequence for each $n$.

$$0 \xrightarrow[]{} M(H_n)_p^- \xrightarrow[]{} \prod_{v\in S_f} \mathbb{Z}_p[G_n]^-/(1-\sigma_{v,n})\xrightarrow[]{} (Y_{S_f,n})_p^- \xrightarrow[]{} 0$$
Here $Y_{S,n}$ is the module $Y_S$ as in the sequence \ref{nabla-sequences} at the layer $H_n$. Now, by taking the projective limit of the above short exact sequence of compact modules under the transition maps induced by Galois restriction, we obtain,

\begin{equation}\label{M-SES}
    0 \xrightarrow[]{} M(H_{\infty})_p^- \xrightarrow[]{} \prod_{v\in S_f} \mathbb{Z}_p[[\mathcal{G}]]^-/(1-\sigma_{v,\infty})\xrightarrow[]{} (Y_{S_f,\infty})_p^- \xrightarrow[]{} 0
\end{equation}
where $(Y_{S_f,\infty})_p=\varprojlim\limits_n (Y_{S_f,n})_p$. Now, observe that all the above modules are finitely generated. For each $v\in S_f$, one can choose $\sigma_{v,\infty}$ such that it has infinite order. Hence, $$\sigma_{v,\infty}=(g_v,\gamma^{k_v})\in G\times \Gamma \cong \mathcal{G}$$ where $\gamma$ is a topological generator of $\Gamma$. Then, $\gamma^{k_v r_v}-1$ is a multiple of $\sigma_{v,\infty}-1$ where $r_v=ord(g_v)$. Hence, the middle module is annihilated by $\prod_{v\in S}(\gamma^{k_v r_v}-1)\in \Lambda$ and therefore it is a $\Lambda-$torsion module. Consequently, so is $M(H_{\infty})_p^-$. Moreover, the transition maps $M(H_{n+1})_p^- \xrightarrow[]{}M(H_n)_p^- $ are surjective. Therefore, the desired claim follows from Corollary 2.3 of \cite{gampheera-popescu-RW}. This completes the proof.   
\end{proof}

\begin{remark}
    Instead of proving the theorem above, one might consider taking the projective limit of Theorem \ref{Fitt-Omega} and proving an equivariant main conjecture for the module $\Omega_{S,S'}^T$ at the top of the Iwasawa tower in the spirit of \cite{gampheera-popescu-RW}. But, unfortunately, unlike the Ritter-Weiss module, the modules $\Omega_{S,S'}^T(H_n)_p$ do not align well in the Iwasawa tower. Check the isomorphism 3.6 in \cite{Atsuta-Kataoka-ETNC}. Therefore, it is not possible to even define an object $\Omega_{S,S'}^T(H_{\infty})_p$.
\end{remark}

Now we are ready to state the main result of this paper. Observe that this generalizes Theorem 4.9 of \cite{gampheera-popescu-RW}.

\begin{theorem}\label{main theorem}
    Suppose that $S$ and $T$ satisfy Assumption \ref{assumption-Iwasawa}. Then, we have the following.
    $$Fitt_{\mathbb{Z}_p[[\mathcal{G}]]^-}(X_S^{T,-})=\Theta_{((S\cap S_{ram})\cup S_p)}^T(H_{\infty})\prod_{v\in S_{ram}\setminus (S\cup T\cup S_p)} Q_v \prod_{v\in S_p\setminus S}\Tilde{J}_v^{\infty}\prod_{v\in S\cap S_{ram}}R_v$$
    Here $S_{ram}=S_{ram}(H_{\infty}/F)$.
\end{theorem}
\begin{proof}
    Let $S'=S_{ram}\setminus (S\cup T)$. Then, by Theorem \ref{equality}, Proposition \ref{LHS} and Proposition \ref{RHS}, we have the following.
    $$Fitt_{\mathbb{Z}_p[[\mathcal{G}]]^-} X_S^{T,-}\prod_{v\in S_f} (\sigma_{v,\infty}-1)=\Theta_{S\cup S_{p}}^T(H_{\infty})\prod_{v\in S'\setminus S_p} Q_v \prod_{S_p\setminus S} \Tilde{J}_v^{\infty}Fitt_{\mathbb{Z}_p[[\mathcal{G}]]^-}(M(H_{\infty})_p^{-}) $$
Now, consider the short exact sequence \ref{M-SES}. We clearly have that $$Fitt_{\mathbb{Z}_p[[\mathcal{G}]]^-}\left(\prod_{v\in S_f}\mathbb{Z}_p[[\mathcal{G}]]^-/(\sigma_{v,\infty}-1)\right)=\left(\prod_{v\in S_f}(\sigma_{v,\infty}-1)\right).$$ As seen in the proof of Proposition \ref{RHS}, the generator is a factor of a nonzero element in the integral domain $\Lambda$. Hence, it is a nonzero divisor of the semi local ring $\mathbb{Z}_p[[\mathcal{G}]]^-$. So, by Proposition 4.9 of \cite{gambheera-popescu}, the middle module of the sequence \ref{M-SES} has projective dimension 1. Then, we have the following.
$$Fitt_{\mathbb{Z}_p[[\mathcal{G}]]^-}M(H_{\infty})_p^-=\prod_{v\in S_f}(\sigma_{v,\infty}-1)Fitt_{\mathbb{Z}_p[[\mathcal{G}]]^-}^{[1]}(Y_{S,\infty})_p^-$$
Now, combining the above equations and cancelling out the principal ideal  $\prod_{v\in S}(\sigma_{v,\infty}-1)$ generated by a nonzero divisor, we obtain the following. Here we are also using the fact that $S'\setminus S_p=S_{ram}\setminus (S\cup T\cup S_p)$.
$$Fitt_{\mathbb{Z}_p[[\mathcal{G}]]^-} X_S^{T,-}=\Theta_{S\cup S_{p}}^T(H_{\infty})\prod_{v\in S_{ram}\setminus (S\cup T\cup S_p)} Q_v \prod_{S_p\setminus S} \Tilde{J}_v^{\infty}Fitt_{\mathbb{Z}_p[[\mathcal{G}]]^-}^{[1]}(Y_{S,\infty})_p^-$$
Observe that, by Theorem A.6 of \cite{gampheera-popescu-RW}, we have
$$Fitt_{\mathbb{Z}_p[[\mathcal{G}]]^-}^{[1]}(Y_{S,\infty})_p^-=\prod_{v\in S\setminus S_{ram}}(\sigma_{v,\infty}-1)^{-1}\prod_{v\in S\cap S_{ram}} R_v$$
We also clearly have $$\Theta_{S\cup S_{p}}^T(H_{\infty})=\Theta_{((S\cap S_{ram})\cup S_{p})}^T(H_{\infty})\prod_{v\in S\setminus S_{ram}}(\sigma_{v,\infty}-1)$$
By putting the above three equations together and canceling out the nonzero divisors gives our result.\end{proof}Observe that $X_{S_{\infty}}^{T,-}=X^{T,-}$. Therefore, by plugging $S=S_{\infty}$ we obtain the following corollary which is an integral equivariant refinement of the Wiles' main conjecture for totally real fields.
\begin{corollary}\label{main corollary}
    Suppose $T\cap S_p=\emptyset$. Then, we have the following.
    $$Fitt_{\mathbb{Z}_p[[\mathcal{G}]]^-}(X^{T,-})=\Theta_{S_p \cup S_{\infty}}^T(H_{\infty})\prod_{v\in S_{ram}\setminus (T\cup S_p)} Q_v \prod_{v\in S_p}\Tilde{J}_v^{\infty}$$
\end{corollary}
\begin{remark}
    One can easily check that for each $v\in S_{ram}$, the ideal generated by $Q_v$ in $\mathbb{Q}_p(\mathbb{Z}_p[[\mathcal{G}]]^-)$ is the whole ring. Same is true for the ideal generated by $\Tilde{J}_v^{\infty}$ for each $v\in S_p$. Therefore, one can recover Theorem \ref{Wiles} from Corollary \ref{main corollary}.
\end{remark}
Based on the Theorem \ref{main theorem} and Corollary \ref{main corollary}, we make the following two conjectures.

\begin{conjecture}
    For any set $S$ of primes, we have the following.
    $$Fitt_{\mathbb{Z}_p[[\mathcal{G}]]^-}(X_S^{-})=\left( \Theta_{((S\cap S_{ram})\cup S_p)}^{\{v_0\}}(H_{\infty})\prod_{v\in S_{ram}\setminus (S\cup \{v_0\}\cup S_p)} Q_v \prod_{S_p\setminus S}\Tilde{J}_v^{\infty}\prod_{S\cap S_{ram}}R_v;p\nmid v_0 \right)$$
\end{conjecture}
\begin{remark}
    Notice that we do not have the option of putting $T=\emptyset$ in Corollary \ref{main corollary} because $T$ is always assumed to be nonempty. Hence, for $T$ the smallest option is $T=\{v_0\}$ where $v_0$ is a prime of $F$ away from $p$. Moreover, we have a surjection $X_S^{\{v_0\},-}\xrightarrow{} X_S^-$. Therefore by the properties of Fitting ideals we have the backward inclusion in the previous conjecture.
\end{remark}
\begin{conjecture}\label{conjecture}
    We have the following.
    $$Fitt_{\mathbb{Z}_p[[\mathcal{G}]]^-}(X^{-})=\left(\Theta_{S_p\cup S_{\infty}}^{\{v_0\}}(H_{\infty})\prod_{v\in S_{ram}\setminus (\{v_0\}\cup S_p)} Q_v \prod_{S_p}\Tilde{J}_v^{\infty};p\nmid v_0 \right)$$
\end{conjecture}
We conclude by computing the Fitting ideal of another interesting equivariant Iwasawa module. 

\begin{theorem}
    Define $\mathcal{A}^{T,\vee,-}:=\varprojlim_n A^T(H_n)^{\vee,-}$ where the transition maps induced by the natural maps $A^T(H_n)^{-}\xrightarrow[]{} A^T(H_{n+1})^{-}$ induced by inclusion. Suppose $T\cap S_p=\emptyset$. Then we have the following.
    $$Fitt_{\mathbb{Z}_p[[\mathcal{G}]]^-}(\mathcal{A}^{T,\vee,-})=\Theta_{S_p \cup S_{\infty}}^T(H_{\infty})\prod_{v\in S_{ram}\setminus (T\cup S_p)}\left(N(\mathcal{I}_v), 1-e_{\mathcal{I}_v}\sigma_{v,\infty}^{-1}\right)$$
\end{theorem}
\begin{proof}
  By Lemma 2.9 of \cite{Greither-Popescu}, the map $A^T(H_n)^{-}\xrightarrow[]{} A^T(H_{n+1})^{-}$ is injective. Then, the transition map between the Pontryagin duals is surjective. Moreover, as noticed in the proof of Corollary 3.15 of \cite{gambheera-popescu}, $\mathcal{A}^{T,\vee,-}$ is a finitely generated and torsion $\Lambda-$module. Therefore, by Corollary 2.3 of \cite{gampheera-popescu-RW} we have the following.

  $$Fitt_{\mathbb{Z}_p[[\mathcal{G}]]^-}(\mathcal{A}^{T,\vee,-})=\varprojlim_n Fitt_{\mathbb{Z}_p[G_n]^-}(A^T(H_n)^{\vee,-})$$

  Now, let us look at the generators of $Fitt_{\mathbb{Z}_p[G_n]^-}(A^T(H_n)^{\vee,-})$ in Theorem \ref{strong-kurihara} at each level. Observe that for each $n$ and a prime $v$ of $F$,
  $$\pi_n^{n+1}(1-e_{I_{v,n+1}}\sigma_{v,n+1}^{-1})=1-e_{I_{v,n}}\sigma_{v,n}^{-1}$$ and $$\pi_n^{n+1}(N(I_{v,n+1}))=p \cdot N(I_{v,n}) \text{ or } N(I_{v,n})$$ depending on whether $p|v$ or not. Hence, when applying Lemma \ref{ignore-generators}, we can ignore the $N(I_{v,n})$ terms when $p|v$. Moreover, since $$\Theta_{S_{\infty}}^T(H_{n})\prod_{v\in S_p}(1-e_{I_{v,n}}\sigma_{v,n}^{-1})=\Theta_{S_{\infty}\cup S_p}^T(H_{n})$$ we have the following.
  $$\varprojlim_n (Fitt_{\mathbb{Z}_p[G_n]^-}A^T(H_n)^{\vee,-})=\varprojlim_n \left[\Theta_{S_{\infty}\cup S_p}^T(H_n)\prod_{v\in S_{ram}\setminus (S_p\cup T)}\left(N(I_{v,n}), 1-e_{I_{v,n}}\sigma_{v,n}^{-1}\right)\right]$$
  This together with Lemma 4.7 of \cite{gambheera-popescu} implies the desired result.
\end{proof}

\bibliographystyle{amsplain}
\bibliography{Fitt}
\end{document}